\documentclass{amsart}
\usepackage[utf8]{inputenc}
\usepackage[paperwidth=8.5in, paperheight=11in]{geometry}
\usepackage{amsmath}
\usepackage{amsrefs}
\usepackage{hyperref}
\usepackage{amscd}
\usepackage{url}
\usepackage{amsthm}
\usepackage{amsfonts}
\usepackage{amssymb}
\usepackage{mathrsfs}  
\usepackage{enumerate}
\usepackage{color}
\usepackage{tikz-cd}
\usepackage{mathtools}

\newcommand{\Rb}{\boldsymbol{\mathrm{R}}}
\newcommand{\Lb}{\boldsymbol{\mathrm{L}}}
\newcommand{\Hb}{\mathbb{H}}

\newcommand{\coker}{\operatorname{coker}}

\newcommand{\cont}{\mathrm{cont}}
\newcommand{\crys}{\mathrm{crys}}
\newcommand{\Fcrys}{{\mathrm{crys},F}}

\newcommand{\Z}{\mathbb{Z}}
\newcommand{\E}{\mathscr{E}}
\newcommand{\C}{\mathscr{C}}

\newcommand{\Rlim}{\Rb\mathrm{lim}}
\newcommand{\F}{\mathbb{F}}
\newcommand{\Q}{\mathbb{Q}}

\newcommand{\X}{\mathscr{X}}
\newcommand{\B}{\mathscr{B}}
\renewcommand{\O}{\mathscr{O}}
\renewcommand{\L}{\mathscr{L}}
\newcommand{\Fs}{\mathscr{F}}
\newcommand{\Gs}{\mathscr{G}}

\newcommand{\Y}{\mathcal{Y}}

\newcommand{\Ps}{\mathscr{P}}
\newcommand{\A}{\mathbb{A}}
\newcommand{\et}{\text{\'{e}t}}

\newcommand{\p}{\mathfrak{p}}
\newcommand{\Frac}{\mathrm{Frac}}

\newcommand{\m}{\mathfrak{m}}

\newcommand{\ov}{\overline}
\newcommand{\wt}{\widetilde}
\newcommand{\Spec}{\operatorname{Spec}}
\newcommand{\spe}{\mathrm{sp}}
\newcommand{\Spa}{\operatorname{Spa}}
\newcommand{\Spf}{\operatorname{Spf}}
\newcommand{\Hom}{\operatorname{Hom}}

\newcommand{\End}{\operatorname{End}}

\newcommand{\Pic}{\operatorname{Pic}}
\newcommand{\NS}{\operatorname{NS}}

\newcommand{\Ab}{\mathrm{Ab}}

\newcommand{\M}{\mathscr{M}}
\newcommand{\N}{\mathscr{N}}
\newcommand{\chara}{\operatorname{char}}

\newcommand{\Mod}{\operatorname{Mod}}

\newcommand{\Res}{\operatorname{Res}}

\theoremstyle{plain}
\newtheorem{thm}{Theorem}
\numberwithin{thm}{subsection} 
\newtheorem{prop}[thm]{Proposition}
\newtheorem{lem}[thm]{Lemma}
\newtheorem{cor}[thm]{Corollary}

\theoremstyle{definition}
\newtheorem{defi}[thm]{Definition}

\theoremstyle{remark}

\newtheorem{assumption}{Assumption}
\newtheorem{rmk}[thm]{Remark}

\newtheorem*{statement}{Statement}
\newtheorem{claim}{Claim}
\newtheorem{setup}[thm]{Setup}

\newtheoremstyle{warnings}{}{}{\color{red}}{}{\color{red}\bfseries}{.}{ }{}

\theoremstyle{warnings}

\title{Specialization of N\'{e}ron-Severi groups in characteristic $p$}
\author{Atticus Christensen}
\address{Department of Mathematics, Massachusetts Institute of Technology, Cambridge, MA 02139-4307, USA}
\email{atticusc@mit.edu}
\urladdr{http://math.mit.edu/~atticusc/}
\thanks{This research was supported in part by a grant from the Simons Foundation (\#402472 to Bjorn Poonen).}
\date{October 15, 2018}

\hypersetup{ colorlinks=true, linkcolor=black, filecolor=magenta,
  urlcolor=cyan, }

\begin{document}
\begin{abstract}
   Andr\'{e} and Maulik--Poonen proved that for any smooth proper family $X\to B$ of varieties over an algebraically closed field of characteristic $0$, there is a closed fiber whose N\'{e}ron-Severi group has the same rank as that of the N\'{e}ron-Severi group of the geometric generic fiber. We prove the analogous statement over algebraically closed fields of characteristic $p>0$ which are not isomorphic to $\overline{\F}_p$. Furthermore, we prove that for any algebraically closed field $k$ of characteristic $p>0$ and smooth proper family $X\to B$ of $k$-varieties, there exists a dense open subvariety $U\subseteq B$ and integer $N$ such that for each map $x\colon\Spec k[[t]]\to U$, the $p$-torsion in the cokernel of the specialization map from the N\'{e}ron-Severi group of the pullback of $X$ to the geometric generic fiber of $x$ to the N\'{e}ron-Severi group of the pullback of $X$ to the special fiber of $x$ is killed by $p^N$. Finally, we prove that for a curve $C$ over $k$ and family $\mathscr{X}\to\mathscr{B}$ of smooth $C$-schemes, there exists a dense Zariski open $U\subseteq \mathscr{B}$ such that for a local uniformizer $t$ at any closed point of $C$, the rank of the N\'{e}ron-Severi group jumps only on a $t$-adic nowhere dense set $t$. The crystalline Lefschetz $(1,1)$ theorem of Morrow is a key ingredient in the proofs.
 \end{abstract}
\maketitle
\tableofcontents

\section{Introduction}
For a proper variety $X$ over an algebraically closed field, we let $\NS(X)$ denote the N\'{e}ron-Severi group. Recall $\NS(X)$ is a finitely generated abelian group. Let $\rho(X)$ denote the rank of $\NS(X)$. The main theorem of the paper (for notation used see Section \ref{sec:background}).
\begin{thm}\label{thm:1}
  Let $K$ be an algebraically closed field of characteristic $p$ such that $K$ is not isomorphic to $\ov{\F_p}$. Let $S$ be finite type $K$-scheme and $X\to S$ a smooth morphism. Denote a geometric generic point of $S$ by $\ov{\eta}$. Then there is an $s\in S(K)$ such that $\rho(X_{\ov{\eta}})=\rho(X_s)$.
\end{thm}

   Andr\'{e} \cite{Andre}*{Th\'{e}or\`{e}me 5.2}, and via different method B. Poonen and D. Maulik \cite{MauPoo}*{Theorem 1.1}, proved the analogous result with $K$ above replaced with any algebraically closed field of characteristic $0$. E. Ambrosi \cite{Ambrosi}*{Corollary 1.6.1.3}, inspired by the proof of Y. Andr\'{e}, has independently proved Theorem \ref{thm:1}. Our proof is different and obtains certain intermediary results we believe to be of independent interest. In particular, we prove in Theorem \ref{thm:2} that the exponent of the $p$-torsion in the cokernel of certain specializations of N\'{e}ron-Severi groups in a family of varieties can be bounded only in terms of the family, independent of the particular points chosen. We use this to prove in Theorem \ref{thm:3a} and Theorem \ref{thm:3b} the nowhere density of the jumping locus for certain families of $k((t))$-varieties, and from this deduce finally Theorem \ref{thm:1} in Section \ref{sec:spread}.

\begin{setup}\label{set:1}
  Let $k$ be an algebraically closed field of characteristic $p>0$ and $C$ a smooth, connected $k$-curve. For any $c\in C(k)$, let $\O_{C,c}$ denote the local ring of the structure sheaf $\O_C$ at $c$, $\widehat{\O_{C,c}}$ its completion, and $\O_c=\ov{\widehat{\O_{C,c}}}$ the integral closure of $\widehat{\O_{C,c}}$ in an algebraic closure of $\Frac(\widehat{\O_{C,c}})$. 
\end{setup}
Recall that if $\O$ is a topological ring and $X$ a finite type $\O$ scheme, then there is a natural topology on the set $X(\O)$. See Section \ref{sec:piadic} for more details.

If $X\to S$ is a map of schemes and $R$ is a valuation ring, for any map $x:\Spec R \to S$, we will denote by $X_x$ the pullback of $X$ to $\Spec R$ along $x$, and by $X_{\eta_x}$ the pullback of $X$ to the generic point of $\Spec R$ along $x$. We will denote $X_{\ov{\eta_x}}$ as the pullback of $X$ to a geometric generic point of $\Spec R$ along $x$. Similarly $X_{s_x}$ will denote the pullback to the special point and $X_{\ov{s_x}}$ the pullback to a geometric special point of $\Spec R$ along $x$.
\begin{thm}\label{thm:3a}
  With assumptions and notation of Setup \ref{set:1}, let $\X\to \B$ be a smooth, proper map of smooth, finite type $C$-schemes. Let $\ov\eta$ be a geometric generic point of $C$. Then there exists a dense Zariski open $U\subseteq \B$, such that for all $c\in C(k)$, $U(\O_c)_{\textup{jumping}}\coloneqq\{x\in U(\O_c):\rho(\X_{\ov{\eta_x}})> \rho(\X_{\ov\eta})\}\subseteq U(\O_c)$ is nowhere dense, where $\O_c$ is regarded as a topological ring and $U(\O_c)$ is given the natural topology.
\end{thm}
An similar theorem in characteristic $0$ was proved by D. Maulik and B. Poonen \cite{MauPoo}.

Let $k$ be an algebraically closed field of characteristic $p>0$. Let $\O$ be the integral closure of $k[[t]]$ in an algebraic closure of $k((t))$. Theorem \ref{thm:3b} will give a criterion for when, given $\X\to \B$ a smooth morphism of $k[[t]]$-schemes, it holds that $\B(\O)_{\text{jumping}}\subseteq \B(\O)$ is nowhere dense. Theorem \ref{thm:3a} in the projective case will be proved immediately after the statement of Theorem \ref{thm:3b}. The deduction of the proper case from the projective case occurs in Section \ref{sec:projprop}.

The following theorem is a key input into the proof of Theorem \ref{thm:3a}.
\begin{thm}\label{thm:2}
  Let $k$ be an algebraically closed field of characteristic $p>0$. Let $X\to S$ be a smooth, proper morphism of finite type $k$-schemes. Then there exists a dense Zariski open $U\subseteq S$ and an integer $N\geq 0$ such that for each map $x:\Spec k[[t]]\to U$, each line bundle $\L$ on $X_{s_x}$, and each $n\geq N$, if $p^n\L$ lifts to $X_x$, then $p^N\L$ lifts to $X_x$.
\end{thm}

The proof of Theorem \ref{thm:2} makes essential use of a crystalline Lefschetz $(1,1)$ Theorem proved by M. Morrow in \cite{Morrow:vartate}. It will be proved in Section \ref{sec:proof2} in the projective case, and the general case will be deduced in Section \ref{sec:projprop}.

\begin{cor}\label{cor:2}
  Let $k$ be an algebraically closed field characteristic $p>0$. Let $C$ be a curve over $k$. Let $\X\to \mathscr{S}$ a proper, smooth morphism of finite type $C$-schemes. Then there a dense open $U\subseteq \mathscr{S}$ and integer $N$ such that for all $c\in C(k)$ the following holds: for any finite flat $\widehat{\O_{C,c}}$-algebra $R$ which is an integral domain, any $x\in U_{\widehat{\O_{C,c}}}(R)$, any element $y\in \NS(\X_{\ov{s_{x}}})$, and any $n\geq N$, if $p^{n}y$ is in the image of specialization from  $\NS(\X_{\ov{\eta_x}})$, then $p^Ny$ is in the image of the specialization.
\end{cor}

\begin{proof}
  We take $U$ so that Theorem \ref{thm:2} holds for $\X\to \mathscr{S}$. Then the result follows immediately from Theorem \ref{thm:2} once we note that any $R$ as in the statement of the corollary is abstractly isomorphic to $k[[t]]$ as a $k$ algebra.
\end{proof}

\subsection{Mixed characteristic counterexample}\label{sec:mixcharcounter}

Corollary \ref{cor:2} is surprising, because its mixed characteristic analogue is false. Below we demonstrate a counterexample to the following statement:

\begin{statement}
  Let $K$ be a number field, let $\O_K$ be its ring of integers, and let $C=\Spec \O_K$. Let $\X\to \mathscr{S}$ a proper, smooth morphism of finite type $C$-schemes. Then there a dense open $U\subseteq \mathscr{S}$ such that for any closed point $\mathfrak{p}$ of $C$: there is an integer $N\geq 0$ such that for any finite flat $\O_{K,\mathfrak{p}}$-algebra $R$ which is an integral domain, any $x\in U(R)$, any element $y\in \NS(\X_{\ov{s_{x}}})$, and any $n\geq N$, if $p^{n}y$ is in the image of specialization from  $\NS(\X_{\ov{\eta_x}})$, then $p^Ny$ is in the image of the specialization.
\end{statement}

As a counterexample, let $\O_K=\Z[\zeta_3]$, let $\mathscr{S}=X(3)[1/6]$ the moduli space of elliptic curves with full level $3$ structure (which we require only to ensure that the moduli space is a scheme) over $\Z[1/6]$, and let $\mathscr{E}$ be the universal family of elliptic curves over $\mathscr{S}$, and let $\X=\mathscr{E}\times_{\mathscr{S}}\mathscr{E}$.

If $E$ is an elliptic curve, $\NS(E\times E)\cong \Z^2\oplus \End(E)$ canonically.

Now, for each $n$ let $L_n$ be an imaginary quadratic field and $\O_{L_n}$ be its ring of integers. Let $p$ be any prime greater than $3$ and let $E_n$ be an elliptic curve over some number field $K_n$ containing $\Q(\zeta_3)$, with full $3$-torsion over $K_n$, all of whose endomorphisms are defined over $K_n$ and such that $\End(E)\cong \Z+p^n\O_{L_n}$. The theory of complex multiplication guarantees that $E_n$ has potentially good reduction, thus perhaps enlarging $K_n$, we take $E_n$ has good reduction everywhere. We may then extend $E_n$ to an elliptic scheme $\E_n$ over the ring of integers $\O_{K_n}$.

Let $p$ be a prime of $\O_K$ lying over $p$, and let $\ov{\E_n}=\E_n\otimes \O_{K_n}/p$. Consider the reduction of $\E_n$ mod a chosen prime $\p_n$ over $\O_{K_n}$ over $p$, and denote this reduction $\ov{\E_n}$. There is a specialization map $\End(E_n)\to \End(\ov{\E_n})$. For any elliptic curve over a field of characteristic $p$, the index of its endomorphism ring in a maximal order is coprime to $p$, so the cokernel of $\End(E_n)\to \End(\ov{\E_n})$ contains a copy of $\Z/p^n\Z$ as $\End(E_n)$ has index $p^n$ in the maximal order $\O_{L_n}$.  Let $R_n$ be the completion of $\O_{K_n}$ at $\p_n$. We may find a map $\Spec R_n\to \mathscr{S}$ such that the pullback of $\X$ is isomorphic to the pullback of $\E_n\times \E_n$ along $\O_K\to R_n$. We will denote this pullback again by $\E_n\times \E_n$ and its generic fiber $E_n\times E_n$.

If $Y$ is a variety over a nonalgebraically closed field $k$, $\NS(Y)$ shall mean $\NS(Y\otimes_k L)$ for some choice of algebraically closed field $L$ containing $k$. Now, we have $\NS(E_n\times E_n)\cong \Z^2\oplus \End(E_n)$ and $\NS(\ov{\E_n}\times \ov{\E_n})\cong \Z^2\oplus \End(\ov{\E_n})$ and the specialization map $\NS(E_n\times E_n)\to \NS(\ov{\E_n}\times \ov{\E_n})$ have cokernel which contains a copy of $\Z/p^n\Z$ coming from the compatible specialization map $\End(E_n)\to \End(\ov{\E_n})$. 

Since there are infinitely many points of $\mathscr{S}$ corresponding to these $E_n$ for every prime $\mathfrak{p}$ of $\O_{K}[1/6]$, any dense open $U$ as in the statement will contain all but finitely many of those points for a prime $p$. Therefore, for any $U\subseteq \mathscr{S}$ dense open, there will be a prime $\p$ of $\O_K$ such that for any integer $M$ there is a finite flat $\O_{K,\mathfrak{p}}$ algebra $R$ which is an integral domain and a point $y:\Spec R\to \mathscr{S}$, such that the $p$-xtorsion in the cokernel of $\NS(\X_{\ov{\eta}_x})\to \NS(\X_{\ov{s}_x})$ is not killed by $p^M$: indeed we may take $R=R_n$ and $y$ to be the point $\Spec R_n\to \mathscr{S}$ corresponding to $\E_n\times_{\O_{K,\mathfrak{p}}}\E_n$ as above. This establishes the falseness of the mixed characteristic version of Corollary \ref{cor:2}.

In the mixed characteristic statement, if we bound the ramification of $R$ a version of Corollary \ref{cor:2} does hold.

\section{Notation and conventions}\label{sec:background}
If $X\to Y$ is a map of schemes, and $x:S\to Y$ a map of schemes, the pullback $X\times_Y S$ will denoted by $X_x$. When $S=\Spec k$ for a field $k$, $X_{\ov x}$, will denote the pullback $X$ to some algebraic closure of $k$ along $x:S\to Y$. We will often suppress the chosen algebraic closure. If $s\in S$ is a point, we will denote $X_s$ the pull back of $X$ to the residue field at $s$ and $X_{\ov{s}}$ the pullback to an algebraic closure of that residue field.

For any $R$ is a local ring which is an integral domain with fraction field $F$ and residue field $k$, if $x:\Spec R\to Y$ is a map, then denote by $X_{s_x}$ the pull back of $X$ to $\Spec k$ along $x$, denote by $X_{\ov{s}_x}$ the pull back of $X$ to an integral closure $k$ along $x$, denote by $X_{\eta_x}$ the pullback of $X$ to $K$ along $x$, and denote by $X_{\ov{\eta}_x}$ the pullback of $X$ to an algebraic closure of $K$ along $x$.

\section{Background}\label{sec:NSgroups}

\subsection{N\'{e}ron-Severi groups}
We begin by recalling some results and definitions about N\'{e}ron-Severi groups.

Let $k$ be an algebraically closed field and $X$ a finite type $k$-scheme. Then $\Pic(X)$ denotes the group of line bundles with multiplication given by tensor product. In this paper $\L_1\otimes \L_2$ will also be denoted as $\L_1+\L_2$. Two line bundles $\L_1$ and $\L_2$ on $X$ are algebraically equivalent if there exists a connected finite type $k$-scheme $S$, a line bundle $\L$ on $X\times S$, and two points $s_1,s_2\in S(k)$, such that identifying the fibres over $s_1$ and $s_2$ with $X$, the restriction of $\L$ over $s_1$ is isomorphic to $\L_1$ and the restriction over $s_2$ is isomorphic to $\L_2$. This is clearly an equivalence relation on $\Pic X$. If $\L_1$ and $\L'_1$ are algebraically equivalent and $\L_2$ and $\L_2'$ are algebraically equivalent, then  $\L_1+\L_2$ and $\L_1'+\L_2'$ are algebraically equivalent. Thus, $\Pic(X)$ modulo algebraic equivalence forms a group, which is called the N\'{e}ron-Severi group of $X$, denoted by $\NS(X)$. This is a finitely generated abelian group (see \cite{Neron:NS}*{p. 146 Th\'{e}or\`{e}me 2} or \cite{SGA6}*{XIII.5.1})). We will let $\rho(X)$ denote the rank of $\NS(X)$.

\begin{rmk}\label{rmk:des}
  Now let $R$ be a discrete valuation ring, and $\X$ a smooth, finite type $R$-scheme. Let $X$ denote the generic fiber. Let $\ov\eta$ be a geometric generic point of $\Spec R$, $s$ the special point and $\ov s$ be a geometric special point. Then there is a map $\NS(X_{\ov\eta})\to \NS(X_{\ov s})$. For $\ell\in \NS(X_{\ov\eta})$ such that there is a line bundle $\L$ on $X$ giving $\ell\in \NS(X_{\ov\eta})$, this map may be given explicitly. Let $D$ be a divisor on $X$ such that $\L\cong \O(D)$. Write $D=\sum_i n_i D_i$ for integral divisors $D_i$, and let $\mathscr{D}_i$ be the closure of $D_i$ in $\X$ and $\mathscr{D}=\sum_i n_i \mathscr{D}$. Then the image of $\ell$ in $\NS(X_{\ov s})$ is represented by the restriction of the line bundle $\O(\mathscr{D})$ to $X_{s}$.  
\end{rmk}

\begin{prop}\cite{MauPoo}*{Proposition 3.6}
  Let $B$ be a noetherian scheme with $s,t\in B$ such that $s$ is a specialization of $t$. Let $p=\chara \kappa(s)$. Let $\X\to B$ be a smooth and proper morphism. Then it is possible to choose a homomorphism $\spe_{\ov t,\ov s}:\NS (\X_{\ov t})\to \NS (\X_{\ov s})$
  satisfying
  \begin{enumerate}[(a)]
  \item If $R$ is a discrete valuation ring and $B=\Spec R$ and $s$ and $t$ are the generic and special points of $B$, and $\ell\in \NS(X_{\ov{t}})$ can be represented by $\L$ on $X_t$, then $\spe_{\ov t,\ov s}(\ell)$ as described in Remark \ref{rmk:des}.
  \item If $p=0$, the $\spe_{\ov t,\ov s}$ is injective and $\coker(\spe_{\ov t,\ov s})$ is torsion free.
  \item If $p>0$, $(b)$ holds after tensoring with $\Z[1/p]$.
  \end{enumerate}

  In all cases $\rho(\X_{\ov s})\geq \rho(\X_{\ov t})$.
\end{prop}
The proof of this proposition uses many facts from \cite{SGA6}.

\begin{prop}\cite{MauPoo}*{Corollary 3.9}
  Let  $k\subseteq k'$ be algebraically closed fields. Let $B$ be an irreducible $k$-variety. For a smooth and proper morphism $\X\to B$, $B(k')_{\text{jumping}}\coloneqq \{b\in B(k'): \rho(\X_b)>\rho(\X_{\ov\eta})\} $ is the union of $Z(k')$ where $Z$ ranges over a countable union of Zariski closed $k$-subvarieties of $B$.
\end{prop}

\subsection{Adic lifts of line bundles}

Throughout this section, $A$ will be a complete noetherian local ring with algebraically closed residue field $k$, $\m$ its maximal ideal, and $X\to \Spec A$ a finite type $A$-scheme. We set $A_n=A/\m^{n+1}$ and $X_n=X\otimes_A A/\m^{n+1}$.
\begin{defi}
 Let $\L$ on $X_0$ be a line bundle. We say that $\L$ lifts to $X$ if there is a line bundle $\L'$ on $X$ such that $\L'|_{X_0}\cong \L$. We say that an element of $x\in \NS(X_0)$ lifts to $X$ if there is a line bundle $\L$ on $X$ such that the image of $\L|_{X_0}$ in $\NS(X_0)$ is $x$.
\end{defi}
\begin{defi}
 Let $\L$ on $X_0$ be a line bundle. We say that $\L$ lifts adically to $X$ if for each $n$, there exists a line bundle $\L_n$ on $X_n$ such that  $\L_n|_{X_0}\cong \L$. We say that an element $x\in \NS(X_0)$ lifts adically to $X$, if for each $n$, there exists an line bundle $\L_n$ on $X_n$ such that the image of $\L_n|_{X_0}$ in $\NS(X_0)$ is $x$.
\end{defi}

\begin{assumption}\label{ass:1}
   A map of schemes $X\to S$ sastisfies {\em Assumption \ref{ass:1}}, if $X\to S$ is projective, the Picard scheme $\Pic_{X/S}$ exists, and furthermore, there exists a subgroup scheme $G$ of $\Pic^\tau_{X/S}$ which is smooth over $S$, and a finite $S$-group scheme $F$ and an exact sequence of group schemes $1\to G\to \Pic^\tau_{X/S}\to F$. We say $X\to S$ satisfies {\em Assumption \ref{ass:1} with respect to $m$}, if $mF=0$.
\end{assumption}

\begin{prop}\label{prop:ass1holds}
  Let $X\to S$ be a smooth, projective morphism of noetherian schemes with $S$ reduced. Then there exists a dense open subscheme $U\subseteq S$ such that Assumption $1$ holds for the restriction $X_U\to U$.
\end{prop}
\begin{proof}
  We begin by noting that the hypotheses are sufficient for the Picard scheme to exist. See \cite{FGA-explained:Kleiman} Theorem 9.4.8.
  
  Now, without loss of generality, we may shrink $S$ so that it is integral. Let $\eta\in S$ be the generic point. Let $k(\eta)$ be the residue field at $\eta$. We have an exact sequence $1\to \Pic_{X_\eta/k(\eta),\text{red}}^0\to \Pic_{X_\eta/k(\eta)}^\tau\to C\to 1$ for some finite $k(\eta)$-group scheme $C$. We recall that $\Pic^0_{X_\eta/k(\eta),\text{red}}$ is smooth. We then spread out $\Pic_{X_\eta/k(\eta)}^0$ to a smooth group scheme $G$ and $C$ to a finite group scheme $F$ over some open $U$ containing $\eta$. Indeed, we may spread out in such a way so that there is an exact sequence $1\to G\to \Pic^\tau_{X_U/U}\to F$.
\end{proof}

\begin{prop}\label{prop:lifting-bundles}
  Assume that $X\to \Spec A$ has a section and satisfies Assumption \ref{ass:1}  with respect to $m$. For a line bundle $\L$ on $X_0$, if $\L$ lifts adically to $X$, then $m\L$ lifts to $X$. Similarly, if $x\in \NS(X_0)$ lifts adically to $X$, then $mx$ lifts to $X$.
\end{prop}
\begin{proof}
  Let $G$ and $F$ be as given in Assumption \ref{ass:1}. Choosing a section $\sigma$ of $X\to \Spec A$, for any $A$-scheme $S$ we identify $\Pic_{X/A}(S)$ with isomorphism classes of line bundles on $X\times_A S$ along with a chosen trivialization along $\sigma\times_A S$ (see \cite{FGA-explained:Kleiman} Theorem 9.2.5).
  
 We may prove the two results simultaneously by showing that if we have a line bundle $\L$ on $X_0$ and line bundles $\L_n$ on each $X_n$ such that $\L_n|_{X_0}$ has the same class as $\L$ in $\NS(X_0)$, then there is a lift of $m\L$. We thus assume we are in this situation.

  Giving a lift of a line bundle on $X_0$ to $X$, by Grothendieck's algebraization theorem, is the same giving as a compatible system of line bundles on each $X_n$ lifting the chosen line bundle.by a line bundle for which the corresponding map

  Let $\M_0=m\L_0$ on $X_0$.

  We will inductively define $\M_n $ on $X_n$ such that $\M_{n}|_{X_{n-1}}\cong \M_{n-1}$ and such that the line bundle $\M_n-m\L_n$ corresponds to a map $\Spec A/\m^{n+1} \to \Pic_{X/A}$ which factors through $G$.

  Take $\M_0=m\L_0$. Assume $\M_k$ has been defined for $k\leq n$. We have $\L_{n+1}|_{X_n}- \L_n$ is in $\Pic^\tau_{X/A}$ as the restriction to $X_0$ has value $0$ in $\NS(X_0)$. Thus $m(\L_{n+1}|_{X_n}- \L_n)$ lies in the image of $G(A_n)$. Thus $\M_n-m\L_{n+1}|_{X_n}$ lies in the image of $G(A_n)$. As $G$ is smooth, we extend this map to a map $\Spec A/\m^{n+2}\to G$. Let $\M_{n+1}'$ be the corresponding line bundle, and let $\M_{n+1}=\M_{n+1}'+m\L_{n+1}$. We then have that $\M_{n+1}$ restricts to $\M_n$ on $X_n$, and that the difference of $\M_{n+1}$ and $m\L_{n+1}$ comes from a map $\Spec A/\m^{n+2}\to G$.

  We have thus constructed a compatible system of line bundles $\M_n$ which restrict to $m\L$ on $X_0$, completing the proof.
\end{proof}

\begin{rmk}\label{rmk:lifting-p-power}
  Let $\L$ be a line bundle on $X_0$. Suppose that $\chara k=p>0$. As we will now explain, for any line bundle $\L$ on $X_0$ and any $n\geq 0$, the line bundle $p^n\L$ lifts to $X_n$. Indeed, this can be seen inductively, noting that if $\M$ is a line bundle on $X_n$, the obstruction to lifting $\M$ to $X_{n+1}$ lives in $H^2(X_0,\O_{X_0})\otimes_k \m^{n+1}/\m^{n+2}$. This obstruction is additive in the line bundle. As $k$ is of characteristic $p$, the obstruction to lifting $p\M$ vanishes, and thus $p\M$ lifts to $X_{n+1}$. 
\end{rmk}
\begin{defi}
  Assume $\chara k=p>0$. Let $x \in \NS(X_0)\otimes \Z_p$. We say that $x$ lifts to  $X_n$ if there exists a line bundle $\L$ on $X_n$ whose image in $\NS(X)/p^{n}$ is equal to that of $\ell$. We say that $x\in \NS(X_0)\otimes \Z_p$ lifts adically to $X$ if it lifts to $X_n$ for each $n$.
\end{defi}

\begin{rmk}
  By Remark \ref{rmk:lifting-p-power}, whether $x\in \NS(X_0)\otimes \Z_p$ lifts to $X_n$ only depends on the image of $x$ in $\NS(X_0)/p^n$. 
\end{rmk}
\begin{defi}
  Assume $\chara k=p>0$.  Let $x \in \NS(X_0)\otimes \Z_p$. We say that $x$ lifts to $X$, if $x$ is in the image of $\varprojlim_n (\Pic(X_n)\otimes \Z_p)\to \NS(X_0)\otimes \Z_p$.
\end{defi}

Similar to Proposition \ref{prop:lifting-bundles}, we have the following:
\begin{prop}\label{prop:compatible-padic-bundles}
  Assume $X\to \Spec A$ has a section and satisfies Assumption \ref{ass:1} with respect to $m$. If $x\in \NS(X_0)$ lifts adically to $X$, then $mx$ lifts to $X$.
\end{prop}
\begin{proof}

Again, let $G$ and $F$ be as given in Assumption \ref{ass:1}. Choosing a section $\sigma$ of $X\to \Spec A$, for any $A$-scheme $S$ we identify $\Pic_{X/A}(S)$ with isomorphism classes of line bundles on $X\times S$ along with a chosen trivialization along $\sigma\times_A S$ (see \cite{FGA-explained:Kleiman} Theorem 9.2.5)

  For all $n$, let $\L_n$ be a line bundle on $X_n$ whose image in $\NS(X_0)/p^{n+1}$ is that of $x$. 

  As $\NS(X_0)$ is a finitely generated abelian group, there exists an $N$ such that for $n\geq N$, $p^n(\NS(X_0)\otimes \Z_p)$ is free. Choose $y_1,\ldots, y_t$ of $p^N\NS(X_0)$ which map to a basis of $p^N(\NS(X_0)\otimes\Z_p)$. For $n\geq N$, let $\N_{n,j}$ be a line bundle on $X_n$ lifting $p^{n-N}y_j$. This is possible by Remark \ref{rmk:lifting-p-power}. There are unique $a_{n,j}\in \Z_p$ such that the image of $\L_n +\sum_j a_{n,j}\N_{n,j}$ in $\NS(X_0)\otimes \Z_p$ is $x$.

  For $n\geq N$, we will inductively define $\M_n$ such that $\M_n|_{X_{n-1}}\cong \M_{n-1}$ and $\M_n-m(\L_n+\sum_j a_{n,j}\N_{n,j})$ is in the image of $G(A_n)\otimes \Z_p$. 
  As a base case we take $\M_N=m\L_N$. Assume that $\M_k$ has been defined for all $N\leq k\leq n$.

  Setting $b_{n,j}=pa_{n+1,j}-a_{n,j}$, we have $\L_{n+1}|_{X_n}+\sum_{j}b_{n,j}\N_{n,j}$ and  $\L_n$ have the same image in $\NS(X_0)$. Therefore, $m\L_n -m(\L_{n+1}|_{X_n}+\sum_{j}b_{n,j}\N_{n,j})$, $\L_n'-m(\L_n+\sum_j a_{n,j}\N_{n,j})$, and $mp\N_{n,j}-m\N_{n+1,j}|_{X_n}$ are all in the image $G(A_n)$. Thus, we find an element $D_{n+1}'\in \Pic(X_n)\otimes \Z_p$ in the image of $G(A_n)\otimes \Z_p$ equal to \[\left(\L_n'-m(\L_n+\sum_j a_{n,j}\N_{n,j})\right)    +\left(m\L_n -m(\L_{n+1}|_{X_n}+\sum_{j}b_{n,j}\N_{n,j})\right)+m\sum_{j}a_{n+1,j}(\N_{n+1,j}|_{X_n}-p\N_{n}).\]
  We lift $D_{n}'$ to $D_{n+1}\in \Pic(X_{n+1})\otimes \Z_p$, and set $\M_{n+1}=m\L_{n+1}+\sum_j a_{n+1,j}\N_{n,j+1}+D_{n+1}$. By construction, $\L_{n+1}'|_{X_n}\cong \L_n'$ and $\M_{n+1}-m(\L_{n+1}+\sum_ja_{n,j}\N_{n,j})$ is in the image of $G(A_{n+1})\otimes \Z_p$.

  Thus we get a compatible system $\{\L_n'\}$ giving the desired element in $\varprojlim_n \Pic(X_n)\otimes \Z_p$.
\end{proof}

\begin{rmk}\label{rmk:p-adic-Picard}
  Let $x\in \NS(X_0)\otimes \Z_p$. For every $n$, choose a line bundle $\M_n$ on $X_0$ which has the same image as $x$ in $\NS(X_0)/p^{n}$. Consider the connected component of the Picard scheme $\Pic_{X_n/A_n}$ containing the point corresponding to $\M_n$. As all $p^n$th powers of line bundles lift to $X_n$, the isomorphism class of this component as an $A_n$-scheme with $\Pic^0$-action does not depend on choice of $\M_n$. Call the component $P_n$. Now, $\M_n$ and $\M_{n+1}$ have the same image in $\NS(X_0)/p^{n}$. Thus by Remark \ref{rmk:lifting-p-power}, we may find $\L_n$ on $X_n$ whose image in $\NS(X_0)$ is the same as the image of $\M_n-\M_{n+1}$. The map sending a line bundle $\L$ on $X_{n}$ to $\L+\L_n$ identifies $P_{n+1}\otimes_{A_{n+1}} A_n$ with $P_n$. We let $\Pic^x_{X/A}$ be the formal scheme associated to this system of $P_n$. That is $\Pic^x_{X/A}=\varinjlim_n P_n$ where the transition maps are the ones described above. The formal scheme $\Pic^x_{X/A}$ up to isomorphism is independent the choices of line bundles $\M_n$ and $\L_n$.
\end{rmk}

We will need the following proposition when working with these $\Pic^x_{X/A}$
\begin{prop}\label{prop:liftingformal}
  Let $A=k[[t,x_1,\ldots,x_n]]$ and let  $f:\X\to \Spf A$ be a proper map of formal schemes. Assume furthermore that the scheme-theoretic image of $\X $ in $\Spf A$ is all of $\Spf A$. Then for any $\Spf k[[t]]$-morphism $y:\Spf k[[t]]\to \Spf A$ there exists a finite extension $R$ of $k[[t]]$, such that there is a $\Spf k[[t]]$-morphism $\Spf R\to \X_y=\X\times_{\Spf A,y}\Spf k[[t]]$. (All rings mentioned are given their ideal of definition given by their maximal ideal).
\end{prop}
If everything were a scheme, properness and scheme-theoretic image being all of $\Spec A$, would imply the map $\X\to \Spec A$ was surjective, from which the claim would easily follow. The following proof uses adic spaces to carry this over to formal schemes.
\begin{proof}
  Let $\mathfrak{X}$ be the adic space associated to $\X$ and let $\mathfrak{S}$ be the adic space associated to $\Spec A$. The map $f:\mathfrak{X}\to \mathfrak{S}$ is proper in the sense of adic spaces.

  The module $f_*\O_{\mathfrak{X}}$ is coherent (\cite{Huber}[Remark 1.3.17]) and the map $\O_{\mathfrak{S}}\to f_*\O_{\mathfrak{X}} $ is injective by the scheme-theoretic surjectivity of $f$. In particular the support of this pushforward is all of $\mathfrak{S}$. As the image of $f$ is closed this means that $f$ must be surjective. Thus for any morphism $\Spf k[[t]]\to \Spf A$, there is a map $\Spa k[[t]]\to \mathfrak{S}$ and thus a map $x:\Spa k((t))\to \mathfrak{S}$. Using the surjectiveness of $f$, we may find a finite field extension $K$ of $k((t))$ and a map $y:\Spa K\to \mathfrak{X}$ such that $y$ maps to $x$. Let $\O_K$ be the ring of integers of $K$. By the properness of $f$, we may find a map $\Spa \O_K\to \mathfrak{X}$ over $\Spa k[[t]]\to \mathfrak{S}$. This gives us a map $\Spf \O_K\to \mathfrak{X}$ over $\Spf k[[t]]\to \Spf A$. Taking $R=\mathfrak{O}_K$, we have proved the proposition.
 
\end{proof}

We now give the refinement of Theorem \ref{thm:2}.
\begin{thm}\label{thm:2prime}
Let $k$ be an algebraically closed field of characteristic $p$. Let $X\to S$ be a smooth, projectve morphism of finite type $k$-schemes. Then there exists a dense Zariski open $U\subseteq S$ and an integer $N\geq 0$ such that for each map $x:\Spec k[[t]]\to U$, each $y\in \NS (X_{s_x})\otimes\Z_p$, and each $n\geq N$, if $p^{n}y$ lifts to $X_x$, then $p^N y$ lifts to $X_x$.
\end{thm}
This will be proved in Section \ref{sec:proof2}. Theorem \ref{thm:2prime} can also be proved with ``projective'' replaced by ``proper'' deducing the proper case from the projective case via the techniques in Section \ref{sec:projprop}. This is not needed for the present applications.

\begin{rmk}
  In the course of the proof of Theorem \ref{thm:2prime}, it will be seen that $U$ can be taken to be an open subset on which Assumption \ref{ass:1} holds and the Newton polygon of $H^2_\crys$ of the fibers of $X$ over points in $U$ are all the same. 
\end{rmk}

Just as Theorem \ref{thm:2} implies Corollary \ref{cor:2}, we have Theorem \ref{thm:2prime} implies the following.
\begin{cor}\label{cor:2prime}
  Let $k$ be a perfect field of characteristic $p$. Let $C$ be a curve over $k$. Let $\X\to \mathscr{S}$ a proper, smooth morphism of finite type $C$ schemes. Then there an open $U\subseteq \mathscr{S}$ and integer $N$ such that for all $c\in C(k)$ the following holds: for any finite $\widehat{\O_{C,c}}$-algebra, $R$, which is an integral domain, $x\in U_{\widehat{\O_{C,c}}}(R)$, and any element $y\in \NS(\X_{s_{x}})\otimes \Z_p$, and any $n\geq N$, if $p^{n}y$ lifts to $\X_x$, then $p^ny$ lifts to $\X_x$.
\end{cor}

\begin{rmk}\label{rmk:2prime}
  If $X\to S$ has a section, Theorem \ref{thm:2prime} and Corollary \ref{cor:2prime} both also hold when ``lifts'' is replaced by ``lifts adically''. This follows from Proposition \ref{prop:compatible-padic-bundles}
\end{rmk}

\begin{assumption}\label{ass:2}
  If $k$ is a perfect field of characteristic $p$, we say a smooth, projective morphism $Y\to S$ satisfies Assumption \ref{ass:2} if the conclusion Theorem \ref{thm:2prime} with $U=S$

  If $m$ is an integer, we say $Y\to S$ satisfies Assumption \ref{ass:2} if the  conclusion of Theorem \ref{thm:2prime} holds with $U=S$ and $N=m$.
\end{assumption}

In Assumption \ref{ass:2} there is no finite type hypothesis on $S$. If $Y\to S$ satisfies Assumption \ref{ass:2}, and $S'\to S$ is a $k$-morphism, then the pullback $Y\times_S S'\to S'$ also satisfies Assumption \ref{ass:2}.

\subsection{$\pi$-adic topology}\label{sec:piadic}

Let $K$ be a nonarchimedean valued field and let $\O_K$ be its valuation ring. We have $\A^n(K)\cong K^n$ and we topologize this via the product topology. If $X$ is a finite type affine $K$-scheme, we may topologize $X(K)$ by first embedding $X\hookrightarrow \A^n$ for some $n$, and giving $X(K)\subseteq \A^n(K)$ the subspace topology. This is independent of embedding. Then for a general finite type $K$-scheme $X$, we topologize $X(K)$ by covering $X$ by affine opens and gluing the topology defined on each of these.

Now assume for the remainder of this section that $\O_K$ is a complete discrete valuation ring with uniformizer $\pi \in \O_K$. The algebraic closure $\ov{K}$ is a field with a canonical topology compatible with that on $K$. Let $\O_{\ov{K}}$ be ring of integers of $\ov{K}$, the integral closure of $\O_{K}$ in $\ov{K}$.

Let $\X$ be a finite type, separated $\O_K$ scheme, and let $X=\X\otimes_{\O_K}K$. We note $\X(\O_{\ov{K}})\subseteq X(\ov K)$ is an open subset. Let $L$ be a finite extension of $K$ contained in $\ov K$ and let $x\in \X(\O_L)$ be such that $\X$ is smooth at every point in the image of $x:\Spec \O_L\to \X$. The completion of $\O_{\X\otimes_{\O_K}\O_L}$ along $x$ is isomorphic to the power series ring $\O_L[[t_1,\ldots,t_n]]$ where the $t_i$ functions defined on some open neighborhood $U$ of the section given by $x$ in $\X_{\O_L}$. Denote this completion by $\widehat{\O_{\X_L,x}}$. Now continuous $\O_L$-algebra homoorphisms $\widehat{\O_{\X_L,x}}\to \ov{K}$ correspond to points $z\in U(\ov{K})$ such that the image of $z$ lands in $U$ the value of each $t_i(z)$ is topologically nilpotent. This is an open set of $X(\ov{K})$ which we denote by $V_x$.

 We now describe a convenient basis of the topology on $V_x$. Let $y\in V_x$ with $y_i=t_i(y)\in \O_{\ov{K}}$. Let $L_y$ be field extension of $L$ such that $y\in \X(\O_{L_y})$. Let $N$ be a positive integer. We consider $V_{y,N}=\{z\in V_x:t_i(z)-t_i(y)\in \pi^N\O_K\}$. As $y$ and $N$ vary, these form a basis of $V_x$. Let $s_i$ denote formal indeterminates and consider the continuous $\O_L$-algebra homomorphism $\widehat{\O_{X_L,x}}\to \O_{L_y}[[s_1,\ldots,s_n]]$ sending $t_i$ to $y_i+\pi_i^N s_i$ and including $\O_L\to \O_{L_y}$. The points of $V_{y,N}$ correspond to continuous $\O_L$-algebra homomorphisms $\widehat{\O_{X,x}}\to \ov{K}$ factoring through $\O_{L_y}[[s_1,\ldots,s_n]]$.

\subsection{Crystalline cohomology}

For definitions and proofs see \cite{Berthelot:cris} or \cite{BO:criscohom}. 

We recall some facts about crystalline cohomology of schemes.

Fix $k$ a perfect field of characteristic $p$. Let $W=W(k)$ be its ring of Witt vectors, and let $\phi:W(k)\to W(k)$ be the Frobenius morphism.. We with crystalline cohomology of $k$-schemes always with respect to the coefficient ring $W(k)$.

Let $F_X:X\to X$ denote the absolute Frobenius on $X$; we have a canonically isomorphism $\phi:\Lb F^*\Rb\Gamma \O_X\to \Rb\Gamma\O_X$, which gives maps $\phi:H^i(X/W)\to H^i(X/W)$ which are Frobenius semilinear with respect to $\phi$ on $W(k)$.

For a $k$-scheme $X$, we will consider the abelian category of $\O_X$-modules on the crystalline site, which will be denote $\Mod(\O_X)$. Additionally, we will consider the quotient of this abelian category by the Serre subcategory consisting of objects killed by a power of $p$. This quotient will be denoted $\Mod(\O_X)_\Q$. Let $\Fs$ be an $\O_X$-module; we will denote its image in $\Mod(\O_X)_\Q$ by $\Fs_\Q$. 

We say a map in $f:M\to N$ in $\Mod(\O_X)$ is an isogeny if there is an integer $n\geq 1$ such that the kernel and cokernel of $f$ are killed by multiplication by $n$. We may also describe $\Mod(\O_X)_\Q$ as the category obtained from $\Mod(\O_X)$ by inverting isogenies.

We will denote the derived category of $\Mod(\O_X)$ by $D_{\crys}(X)$, and we will denote the derived category of $\Mod(\O_X)_\Q$ by $D_{\crys}(X_\Q)$. Note there is an obvious triangulated functor $D_{\crys}(X)\to D_{\crys}(X_\Q)$.

Recall element $M$ of $\Mod(\O_X)$ is a crystal if for every map $f:(U,T,\delta)\to (U',T',\delta')$ in the crystalline site of $X$, the natural map $f^*M_{(U',T',\delta')}\to M_{(U,T,\delta)}$ is an isomorphism. We say that $M$ is coherent if each $M_{(U,T,\delta)}$ is a coherent sheaf. We say $M$ is flat if each $M_{(U,T,\delta)}$ is a flat $\O_T$ module.

For any scheme $X$ of characteristic $p$, we may form its de Rham-Witt $F$-$V$-pro-complex $\{W_r\Omega_X^\bullet\}_{r\geq 1}$. For $X$ smooth over a perfect field, the definition is due to Bloch, Deligne and Illusie; see \cite{Illusie:dRW}. For general $X$ the definition is due to Hesselholt and Madsen for $p$ odd and Costeanu for $p=2$ \cite{Costeanu}; see \cite{Hesselholt:bigDRW} for definition and proof of properties. For any $i$, we define $\mathcal{F}:W_{r+1}\Omega_X^{\geq i}\to W_r\Omega_X^{\geq i}$ degree-wise as $p^vF:W_{r+1}\Omega^{i+v}_X\to W_r\Omega^{i+v}$. Similarly, we define $\mathcal{V}:W_r\Omega^{<i}\to W_{r+1}\Omega^{<i}$ degree-wise $p^{i-1-v}V:W_r\Omega_X^{i-1-v}\to W_{r+1} \Omega_X^{i-1-v}$. Additionally, we define $\phi:W_{r+1}\Omega_X^\bullet\to W_r\Omega_X^\bullet$ degree-wise as $p^vF:W_{r+1}\Omega^{v}_X\to W_r\Omega^{v}$. The restriction of $\phi$ to any truncation of the de Rham-Witt complex will still be denoted by $\phi$. As $W_r\Omega_X^\bullet$ satisfies \'{e}tale descent, we will sometimes view it as a sheaf on the \'{e}tale site.

We will also consider the logarithmic subcomplex $W_r\Omega_{X,\log}^\bullet$, which is a subsheaf of $W_r\Omega_{X,\log}^\bullet$ on the \'{e}tale site of $X$ locally generated as an abelian group by elements on the form $d\log[\alpha]$ for $\alpha\in \O_X^\times$, and where $[\alpha]$ denotes the Teichm\"{u}ller lift to $W_r\O_X=W_r\Omega_{X}^0$ and $d\log[\alpha]$ means $(d[\alpha])/\alpha$.

We also recall that for any pro-system of sheaves $\{\Fs_r\}_r$ on a site $C$ the continuous cohomology $\Hb_\cont^i(C,\{\Fs_r\})$ is the $i$th cohomology group of $\Rlim_r \Rb \Gamma_C(\Fs_r)$ where $\Gamma_C$ is the global section functor (see \cite{Jannsen:contCohom}). The notation $\Hb^n_\cont(X,W\Omega_X^\bullet)$ will denote the continuous cohomology with respect to the pro-system $\{W_r\Omega_X^\bullet\}_{r\geq 1}$, and similarly if $W\Omega_X^\bullet$ is replaced with, $W\Omega_X^i$, truncations $W\Omega_X^{\geq i}$ etc. These are the groups are canonically isomorphic whether we consider the cohomology on the Zariski site of $X$ on the \'{e}tale site of $X$. The cohomology of $W_r\Omega_{X,\log}^\bullet$ will always be taken on the \'{e}tale site of $X$, and $H^n(X_\et,W\Omega^i_{X,\log})$ will denote the continuous cohomology of $\{W_r\Omega_{X,\log}^i\}$ when considered as a sheaf on the \'{e}tale site of $X$.

Section 3.1 of \cite{Morrow:vartate} shows that for any regular $k$-scheme $X$ and any $n\geq 0$, there is a canonical isomorphism $H^n_\crys(X/W)\cong \Hb_\cont^n(X,W\Omega_X^\bullet)$. The natural maps $W_r\Omega^i_{X,\log}[-i]\to W_r\Omega^\bullet_X$ induce a map $H^i_\cont(X,W_r\Omega^i_{X,\log})\to \Hb^{2i}_\cont(X, W_r\Omega_X)\cong H^{2i}_\crys(X/W(k))$. If $\L$ is a line bundle on $X$, there is a Chern class $c_1(\L)\in H^1_\cont(X,W\Omega_{X,\log})$. This maps to the usual crystalline Chern class in $H^2_\crys(X/W)$.

Now let $A$ be a complete, noetherian, local ring. Assume that $A$ is of characteristic $p$ and is $F$-finite, which means that $A$ is finite over its subalgebra of $p$th powers. Let $X$ be an $A$-scheme. By functoriality, we get a map $\Rb \Gamma W_r\Omega^\bullet_{X,\log}\to \Rb \Gamma W_r\Omega^\bullet_{X_n,\log}$. Taking inverse limits over $r$ gives $\Rlim_r\Rb \Gamma W_r\Omega_{X,\log}^\bullet\to \Rlim_r\Rb \Gamma W_r\Omega^\bullet_{X_n,\log}$, then $\Rlim_r\Rb \Gamma W_r\Omega^\bullet_{X,\log}\to \Rlim_n \Rlim_r\Rb \Gamma W_r\Omega^\bullet_{X_n,\log}$. Taking cohomology, gives the first map in the following  composition \[H^i_\cont (X,W\Omega_{X,\log})\to H^i(\Rlim_n \Rlim_r\Rb \Gamma W_r\Omega_{X_n,\log})\to \varprojlim_n H^i_\cont(X_n, W\Omega_{X_n,\log}).\]

\begin{prop}\label{prop:surj}
  With $A$ and $X$ as above and $X\to \Spec A$ proper, for all $i$ and $j$ the maps $H^i_\cont (X,W\Omega^j_{X,\log})\to \varprojlim_n H^i_\cont(X_n, W\Omega_{X_n,\log}^j)$ are surjective.
\end{prop}
\begin{proof}
  By \cite{Morrow:KlogHW} Corollary 4.1, for any scheme $Y$ in characteristic $p$, we have an exact sequence of pro-\'{e}tale sheaves \[0\to \{W_r\Omega_{Y,\log}^j\}_r\to \{W_r\Omega_{Y}^j\}_r\xrightarrow{1-\mathcal{F}} \{W_r\Omega_{Y}^j\}_r\to 0\]

  Functoriality of all terms involved then gives us morphisms of distinguished triangles
{\small
  \begin{center}
    \begin{tikzcd}
      \Rlim_r \Rb\Gamma (X,W_r\Omega^j_{X,\log})\arrow[r]\arrow[d] &\Rlim_r \Rb\Gamma (X,W_r\Omega^j_{X})\arrow[r,"1-F"]\arrow[d]&\Rlim_r \Rb\Gamma (X,W_r\Omega^j_{X})\arrow[d]\rightarrow \\
      \Rlim_n\Rlim_r \Rb\Gamma (X_n,W_r\Omega^j_{X_n,\log})\arrow[r] &\Rlim_n\Rlim_r \Rb\Gamma (X_n,W_r\Omega^j_{X_n})\arrow[r,"1-F"]&\Rlim_n\Rlim_r \Rb\Gamma (X_n,W_r\Omega^j_{X_n})\rightarrow
    \end{tikzcd}
  \end{center}
}
  The second and third vertical arrows are induced by the maps $\Rb\Gamma (X,W_r\Omega_{X}^j)\to \Rlim_n  \Rb\Gamma (X_n,W_r\Omega_{X_n}^j)$. That this map is an isomorphism follows from formula (18) in \cite{Morrow:KlogHW}, which in turn comes from an application of the formal function theorem. Thus the above map of distinguished triangles is an isomorphism, and therefore the first vertical arrow is an isomorphism.

  Using the composition of functors spectral sequence and the fact that in the category of abelian groups $\varprojlim^t=0$ for $t\geq 2$, we get short exact sequences for each $i$: \[0\to \sideset{}{^{1}}\varprojlim_n H^{i-1}_\cont (X,W\Omega^j_{X,\log})\to H^i_\cont(X,W\Omega^j_{X,\log})\to \varprojlim_n H^i_\cont(X_n,W\Omega_{X_n,\log}^j)\to 0,\] where the last arrow is the one induced by functoriality. This finishes the proof.
  
\end{proof}

\begin{rmk}\label{rmk:crystals}
  Let $S$ be a smooth, irreducible, affine finite type $k$-scheme or $S=\Spec k[[t]]$. We may find a $p$-adically complete,  formally smooth, affine formal scheme $\wt{S}$ whose reduction mod $p$ is $S$ (in the case when $S=\Spec k[[t]]$, we take $\wt{S}=\Spf W(k)[[t]]$). Let $B$ be such that $\wt S=\Spf B$. Let $D$ be the $p$-adically complete formal scheme obtained from $\wt{S}$ by taking the $p$-adic completion of the PD envelope of the diagonal embedding of $S$ in $\wt{S}\times_{W(k)} \wt{S}$. We get projections $\pi_i:D\to \wt{S}$ for $i=1,2$. Giving a crystal, $\C$, on $S$ is equivalent to giving a quasi-coherent sheaf $\mathscr{E}$ on $\wt{S}$ and an isomorphism $\pi_1^*\E\cong \pi_2^*\E$ satisfying a cocycle condition. The crystal is flat if and only if $\E$ is a a flat sheaf on $\wt{S}$ by \cite{BO:criscohom} Theorem 6.6.
\end{rmk}

\begin{rmk}\label{rmk:dualcrystals}
  Let notation be as in Remark \ref{rmk:crystals} and let $C$ be a crystal on $S$ with value $\E$ on $\wt{S}$. We denote $C^\vee$ the following crystal: the value $\wt{S}(\C^\vee)=\Hom(\E,\O_{\wt{S}})$ and the map $\pi_1^*\Hom(\E,\O_{\wt{S}})\to \pi_2^*\Hom(\E,\O_{\wt{S}})$ is that induced by the map $\pi_1^*\E\to \pi_2^* \E$ coming from the fact that $\C$ is a crystal via Remark \ref{rmk:crystals}.

  If $C$ is a flat, coherent crystal, then for any element $(U,V,\delta)$ of the crystalline site of $S$, $C^\vee(U,V,\delta)=\Hom_{\O_V}(C(U,V,\delta),\O_V)$.
\end{rmk}

\begin{lem}\label{lem:flatonopen}
  Let $S$ be a smooth, finite type $k$-scheme. For any coherent crystal $C$ on $S$, there is a dense open $U\subseteq S$ and a flat coherent crystal $C'$ on $U$ such that $C|_U$ and $C'$ are isogenous.
\end{lem}
  \begin{proof}
    We assume that $S$ is irreducible and affine. We then take notation as in Remark \ref{rmk:crystals} including setting $\E=\wt{S}$.

  Taking the value of $C$ at $\wt{S}$, we get a coherent sheaf $\E$ on $\wt{S}$. Let $T\subseteq \E$ be the subsheaf of elements killed by a power of $p$. As $\wt{S}$ is noetherian and our crystal is coherent, there is an integer  $n$ such that $p^nT=0$. Now $\E/T$ is $p$-torsion free. Let $\E'=\E/T$. As the  maps $\pi_i:D\to\wt{S}$ are flat by \cite{BO:criscohom} Corollary 3.35, $\pi_i^*T$ is the subsheaf of $\pi_i^*\E$ of elements killed by a power of $p$, and $\pi_i^*(\E/T)=\pi^*\E/\pi^*T$ is $p$-torsion free. The map $\pi_1^*\E\cong \pi_2^*\E$ then induces an isomorphism $\pi_1^*(\E/T)\cong \pi_2^*(\E/T)$ such that the cocycle condition still holds. Thus by the equivalence described in the first paragraph, we get crystal $C'$ whose value on $\wt S$ is $p$-torsion free and an isogeny $C\to C'$.

  Let $M$ be the $B$-module corresponding to the coherent sheaf $\E'$. By abuse of notation, we also refer to the coherent sheaf on $\Spec B$ corresponding to $M$ as $\E'$. As $B$ is regular, the local ring at the ideal $(p)$ is a discrete valuation ring with uniformizer $p$. As $M$ is $p$-torsion free, the localization $M_{(p)}$ is a flat $B_{(p)}$-module. Thus, over some open subset of $\Spec B$ containing the generic point of $S=\Spec B/p\subseteq \Spec B$ $M$ is flat. Let $U$ be the the restriction of this open to $S$. It follows that $\E'|_U$ is flat. Thus $C'|_U$ is a flat crystal on $U$ isogenous to $C|_U$ and $C'|_U$.
\end{proof}

We make essential use of the following result:
\begin{thm}[\cite{Morrow:KlogHW} Theorem 3.4]\label{thm:vartate}
  Let $X$ be a smooth scheme over $\Spec k[[t]]$. Let $X_0=X\otimes_{k[[t]]}k$. Let $\L$ be a line bundle on $X_0$, and let $c=c_1(\L)$ be its first Chern class in $H^1_\cont(X_0,W\Omega^1_{X_0,\log})$. Then if $\L$ lifts to $X$ if and only if $c$ is in the image of the map \[H^1_\cont(X,W\Omega^1_{X,\log})\to H^1_\cont(X_0,W\Omega^1_{X_0,\log})\]
\end{thm}

A version of this result was first proved in \cite{Morrow:vartate} and reproved in \cite{LazdaPal:11}. For the current application, we prefer the statement found \cite{Morrow:KlogHW}.

\subsection{Modules with $F$-structure}\label{sec:Fstruct}

For this section, let $k$ be a perfect field of characteristic $p>0$ and let $X$ and $Y$ be $k$-schemes. The purpose of this section is to explain how, if $f:X\to Y$ is a $k$-morphism, $\Rb f_{*}\O_X\in D(\O_Y)_{\crys}$ can be represented by a complex of sheaves each with an action of Frobenius.

Let $F:X\to X$ denote the absolute Frobenius on $X$.

\begin{defi}
  Given an element $M$ of $\Mod(\O_X)_\crys$, an $F$ structure on $M$ is a map $F^{-1}M\to M$ compatible with the map $F^{-1}\O_X\to \O_X$.
\end{defi}

In general, let $T$ be a topos (by which we mean the category of sheaves on a site). Let $F:T\to T$ be a map of topoi. Then we may form the category $T_F$ whose objects are pairs $(t,\phi)$ such that $t\in T$ and $\phi:F^{-1}t\to t$. Morphisms from $(t,\phi)$ to $(t',\phi')$ are maps $\psi:t\to t'$ in $T$ such that the following diagram commutes
\[
  \begin{tikzcd}
    F^{-1}t\arrow[r,"\phi"]\arrow[d,"F^{-1}\psi"]& t\arrow[d,"\psi"]\\
    F^{-1}t'\arrow[r,"\phi'"]&t'.
  \end{tikzcd}
\]

\begin{prop}
  If $T$ is a topos with self-map $F:T\to T$, then $T_F$ is also a topos.
\end{prop}
The proof is relatively straightforward, once you assume that $F$ comes from a map of sites $F:\mathscr{C}\to \mathscr{C}$. This can be arranged with judicious use of \cite{stacks-project}*{\href{http://stacks.math.columbia.edu/tag/03A1}{Tag 03A1}}

\begin{rmk}
  There is a functor $T_F\to T$ which forgets the $F$ structure. However, while this functor does have a left adjoint, this adjoint does not preserve finite limits. Therefore, this natural functor is not a map of topoi.
\end{rmk}

A ring object $R$ in $T$ with a ring object map $F^{-1}R\to R$ gives a ring in $T_F$, which we call $R_F$. Let $\Mod(R)$ denote the category of modules over $R$ and $\Mod(R_F)$ the category of modules over $R_F$. An element of $\Mod(R_F)$ is an $R$-module $M$ with a map $F^{-1}M\to M$ compatible the map $F^{-1}R\to R$. Thus there is a forgetful map $\Mod(R_F)\to \Mod(R)$. Like above, this map does not come from a map of ringed topoi. Let $\Ab(T)$ be the category of abelian group objects on $T$, and similarly define $\Ab(T_F)$. We have a functor $\Ab(T_F)\to \Ab(T)$

We recall
\begin{defi}\cite{stacks-project}*{\href{http://stacks.math.columbia.edu/tag/079X}{Tag 079X}}
  Let $C$ be a site. An abelian sheaf $\mathcal{F}$ on $C$ is said to be limp if for every sheaf of sets $K$, we have that $H^p(K,\mathcal{F})=0$ for all $p\geq 1$.
\end{defi}

The importance of limp sheaves comes from the following proposition.
\begin{prop}
  If $f:T\to T'$  is a morphism of topoi, limp sheaves are acyclic for $f_*:\Ab(T)\to \Ab(T')$.
\end{prop}

\begin{prop}\label{prop:limp}
  With notation as above, the functor $\Ab(T_F)\to \Ab(T)$ has an exact left adjoint, takes injective objects to injective objects and limp sheaves to limp sheaves. Additionally, $\Mod(R_F)\to \Mod(R)$ takes limp sheaves to limp sheaves and in particular takes injective sheaves to limp sheaves. Lastly, $\Mod(R_F)$ has enough injectives.
\end{prop}
\begin{proof}
  The left adjoint is given by taking an $A\in \Ab(T)$ to $\bigoplus_{i=1}^\infty F^{-i}A$ with $F$-structure given by the composite $F^{-1}(\bigoplus_{i=1}^\infty F^{-i}A)=\bigoplus_{i=2}^\infty F^{-i}A\hookrightarrow \bigoplus_{i=1}^\infty F^{-i}A$. This is obviously exact
\end{proof}

\begin{rmk}\label{rmk:Fstruct}

We may apply this all to the case of $\Mod(\O_X)_\Fcrys$. Let $X_{\Fcrys}$ denote the ringed topos constructed via the above process taking $T=X_\crys$ and $F:X_{\crys}\to X_{\crys}$ to be map induced from $F_X:X\to X$. This is even a map of ringed sites if we give $X_{\crys}$ the structure ring $\O_X$; there is a natural map $F:F^{-1}\O_X\to \O_X$.

Let $f:X\to S$ be a map of $k$ schemes. Let $F_X$ denote the absolute Frobenius on $X$ and similarly $F_S$ denote the absolute Frobenius on $S$. We have $f\circ F_X=f\circ F_S$. We obtain a map of ringed topoi $f_F:X_\Fcrys\to S_\Fcrys$. We also obtain a commutative diagram:

\[
  \begin{tikzcd}
    \Mod(\O_X)_\Fcrys\arrow[r]\arrow[d,"f_{F,*}"]& \Mod(\O_X)_\crys \arrow[d,"f_{*}"]\\
    \Mod(\O_S)_\Fcrys\arrow[r]& \Mod(\O_S)_\crys.
  \end{tikzcd}
\]

Let $D(\O_X)_\Fcrys$ denote the derived category of $\Mod(\O_X)_\Fcrys$. We thus get a diagram
\[
  \begin{tikzcd}
    D(\O_X)_\Fcrys\arrow[r]\arrow[d,"\Rb f_{F,*}"]& D(\O_X)_\crys \arrow[d,"\Rb f_*"]\\
    D(\O_S)_\Fcrys\arrow[r]& D(\O_S)_\crys.
  \end{tikzcd}
\]

Proposition \ref{prop:limp} implies this diagram commutes.

\end{rmk}
\begin{defi}\label{defi:Fstruct}
  Let $X$ be a $k$-scheme. Let $\Mod(\O_X,F)_\Q$ be the abelian category obtained form $\Mod(\O_X,F)$ by inverting multiplication by $p$, or equivalently by quotienting by the Serre subcategory consisting of modules killed by a power of $p$. For any $(M,\phi)\in \Mod(\O_X,F)_\Q$, define $M(n)$ to be the element of $\Mod(\O_X,F)_\Q$ given by $(M,p^{-n}\phi)$.
\end{defi}
\section{Proof of Theorem \ref{thm:2} in the projective case}\label{sec:proof2}

The following is a partial strengthening of a theorem of Matthew Morrow. Recall $\phi$ refers to the Frobenius action on crystalline cohomology.
\begin{thm}[\cite{Morrow:vartate} 3.5,3.7]\label{thm:logcris}

  Let $X$ be a regular $\F_p$ scheme. Then the cokernel of the map $H^1_\cont(X_{\et},W\Omega_{X,\log})\to H^2_{cris}(X)^{\phi=p}$ is killed by $p^2$. Furthermore,  if $X$ is smooth and proper over a perfect $k$, this map has kernel killed by $p$.

\end{thm}

\begin{proof}

  The above statement is a strengthening of Theorem 3.5 in \cite{Morrow:vartate}; we work out what the constant of Remark 3.7 is. 

  We trace through the proof of Proposition 3.2 of \cite{Morrow:vartate} in the case when $i=1$. We will use intermediary results and notation of that proof.

  We begin with the exact sequence of pro-complexes from \cite{Morrow:vartate} equation $(10)$:
  \[0\to \{W_r\Omega^i_{X,\log}\}_r\to\{W_r\Omega_X^{\geq i}\}_r\xrightarrow{1-\mathcal{F}}\{W_r\Omega_X^{\geq i}\}_r\to 0.\]
  
Taking cohomology gives the exact sequence
\[H^1_\cont(X_\et, W\Omega^1_{X,\log})\to \Hb^{2}_\cont(X,W\Omega_X^{\geq 1})\xrightarrow{1-\mathcal{F}}\Hb^2_\cont(X,W\Omega_X^{\geq 1})\]
  
Thus, $H^1_{\cont}(X_{\et},W\Omega^1_{X,\log})$ surjects onto the kernel of $1-\mathcal{F}$ acting on $\Hb^2_\cont(X,W\Omega_X^{\geq 1})$, and as $p(1-\mathcal{F})=p-\phi$, the map \[ H^1_{\cont}(X_{\et},W\Omega^1_{X,\log}) \to \Hb^2_\cont(X, W\Omega_X^{\geq 1})^{\phi=p}\] has cokernel killed by $p$.

From the sequence of pro-complexes of sheaves \[0\to \{W_r\Omega^{\geq i}_X\}_r\to \{W_r\Omega_X^\bullet\}\to \{W_r\Omega_X^{<i}\}\to 0\] we get the exact sequence:
\[\Hb^{1}_\cont(X,W\Omega_X^{<1})\to \Hb^{2}_\cont(X,W\Omega_X^{\geq 1})\to H^2_\crys(X/W(k))\to \Hb^{2}_\cont(X,W\Omega_X^{<1})\]

Now, in the proof of Proposition 3.2 in \cite{Morrow:vartate} it is proved that $1-\mathcal{V}$ is an isomorphism on $\Hb^n_\cont(X,W\Omega^{>i})$ for all $n$ and $i$. This and two applications of the snake lemma then imply that $\Hb^{2}(X,W\Omega^{\geq 1})^{1=\mathcal{V}}\to H^{2}_{cris}(X)^{1=\mathcal{V}}$ is surjective. As $\phi \mathcal{V}=FV=p$ on $W\Omega_X^{<1}$, we see that the map $\Hb^{2}(X,W\Omega_X^{\geq 1})^{\phi=p}\to H^{2}_{cris}(X)^{\phi=p}$ has cokernel killed by $p$. We conclude the composite map $H^1_{\cont}(X_{\et},W\Omega^1_{X,\log})\to H^{2}_{\crys}(X)^{\phi=p}$ has cokernel killed by $p^2$.

Now, consider $X$ which is smooth and proper over $k$. We have that $\Hb^{1}_{\cont}(X,W\Omega_{X_0}^{<1})=H^1(X_0,W\O_{X_0})$. By II.3.11 of \cite{Illusie:dRW}, the slope spectral sequence degenerates at $H^1(X_0,W\O_{X_0})$ and hence the map $\Hb^{1}_{\cont}(X,W\Omega_{X}^{<1})\to \Hb_\cont^2(X,W\Omega_X^{\geq 1})$ is zero, and thus $\Hb_\cont^{2}(X,W\Omega_{X}^{\geq 1})\to H^{2}_{\crys}(X/W(k))$ is an injection. This in turn implies that $\Hb_\cont^{2}(X,W\Omega_{X}^{\geq 1})^{\phi=p}\to H^{2}_{\crys}(X)^{\phi=p}$ is an injection.

The kernel of $H^1_\cont(X_\et, W\Omega^1_{X,\log})\to \Hb^{2}_\cont(X,W\Omega_X^{\geq 1})$ is the cokernel of $\Hb^1_\cont(X,W\Omega_X^{\geq 1})\xrightarrow{1-\mathcal{F}}\Hb^1_\cont(X,W\Omega_X^{\geq 1})$. This is isomorphic to the map $\Hb^0_\cont(X,W\Omega_X^1)\xrightarrow{1-\mathcal{F}}\Hb^0_\cont(X,W\Omega_X^1)$. Now $\Hb^0_\cont(X,W\Omega_X^1)$ is a torsion-free Dieudonne module with Frobenius action given by $\mathcal{F}$ (see \cite{Illusie:dRW} II.2.19). Let $N=\Hb^0_\cont(X,W\Omega_X^1)/p$. This is a finite dimensional vector space over $k$ on which $\mathcal{F}$ acts semilinearly. If $\mathcal{F}$ acts by injectively it must act bijectively as $k$ is perfect. In this case, Frobenius descent implies that $1-\mathcal{F}$ is surjective. If $\mathcal{F}$ is zero it is obvious $1-\mathcal{F}$ is an isomorphism. Now let $N_1$ be the subvector space on which $\mathcal{F}$ acts by zero. Clearly $N_1$ is in the image of $1-\mathcal{F}$ and $\mathcal{F}$ preserves $N_1$, thus $\mathcal{F}$ descends to an action on $N/N_1$ and we conclude by induction on dimension that $1-\mathcal{F}$ is surjective on $N/N_1$ from which we conclude it is surjective on $N$.

An approximation argument then implies that $\Hb^0_\cont(X,W\Omega_X^1)\xrightarrow{1-\mathcal{F}}\Hb^0_\cont(X,W\Omega_X^1)$ is surjective, thus the cokernel of $\Hb^0_\cont(X,W\Omega_X^1)\xrightarrow{p-\phi}\Hb^0_\cont(X,W\Omega_X^1)$ is killed by $p$.

We now put everything together: for general $X$ we get that the map $H^2_\cont(X_\et,W\Omega^1_{X,\log})\to H^2_\crys(X)^{p=\phi}$ has cokernel killed by $p^2$, and in the case when $X$ is smooth and proper over $k$ the map $H^2_\cont(X_{\et},W\Omega^1_{X,\log})\to H^2_\crys(X)^{p=\phi}$ has cokernel killed by $p^2$ and has kernel killed by $p$.

\end{proof}

\begin{prop}\label{prop:fs}
  Let $k$ be a perfect field of characteristic $p$. Let $S$ be a smooth $k$-variety and $f:X\to S$ be a smooth, projective morphism of $k$-varieties.  There are $s\geq 0$ and $m\geq 0$ such that for each $x:\Spec k[[t]]\to S$, there exist flat crystals $C^i_x$ on $\Spec k[[t]]$ such that
  \begin{enumerate}
  \item there is a map $\Rb_\crys f_* \O_{X,x}\to \bigoplus_i C^i_x[-i]$ whose cone has cohomology sheaves killed by $p^s$
  \item there are maps $p^m\phi:F^{-1}C^i_x\to C^i_x$ making the following diagram commute
    \begin{center}
      \begin{tikzcd}
        F^{-1}\Rb_\crys f_*\O_X\arrow[r,"p^m \phi"]\arrow[d]& \Rb_\crys f_*\O_X\arrow[d]\\
        \bigoplus_i F^{-1}C^i_x\arrow[r,"p^m \phi"]& \bigoplus_i C^i_x.
      \end{tikzcd}
    \end{center}
  \end{enumerate}
\end{prop}

\begin{proof}
  We assume $S$ connected. Let $d$ be the relative dimension of $X\to S$. For any open cover $\{U_i\}$ of $S$, any $\Spec k[[t]]$-point of $S$ will factor through one of the $U_i$ as $k[[t]]$ is local. Thus we may and do assume that $S$ is affine. Let $\widetilde{S}$ be a formal scheme over $W(k)$ which is a $p$-adically complete, formally smooth deformation of $S$, which exists because $S$ is affine. By \cite{Morrow:crystal} Corollary 2.4, each $\widetilde{S}(R^i_\crys f_* \O_X)_\Q$ is a projective $(\O_{\widetilde{S}})\otimes \Q$ module. We further shrink $S$ so that each $\widetilde{S}(R^i_\crys f_* \O_X)_\Q$ is free. There is then an integer $\alpha$ and integers $r_i$ such that there are maps $\widetilde{S}(R^i_\crys f_* \O_X)\to \O_{\widetilde{S}}^{r_i}$ whose kernel and cokernel are killed by $p^\alpha$.

  As in \cite{Morrow:crystal} Lemma 2.2, there is a crystal, which we will denote $\Fs^i$ whose value on $\widetilde{S}$ is $\widetilde{S}(R^i_\crys f_*\O_X)$ and with natural isogeny $\Fs^i\to R^i_\crys f_*\O_X$. For any $U$ an open of $S$ and PD thickening $U\hookrightarrow T$ and any map $h:T\to \widetilde{S}$ compatible with $PD$ structures, $\Fs^i(T)=h^*\widetilde{S}(R^i_\crys f_*\O_X)$. By the basechange theorem for crystalline (\cite{BO:criscohom}[Theorem 7.8]) and a Tor spectral sequence, it follows that the power of $p$ killing the kernel and cokernel of the natural isogeny $\Fs^i\to R^i_\crys f_*\O_X$ is bounded by a constant depending only on $\alpha$.

  By \cite{Morrow:crystal} Proposition 5.1, if we choose a line bundle $\L$ on $X$ relatively ample for $X\to S$, we have a Lefschetz map $L^i:R^{d-i}f_*\O_X\to R^{d+i}f_*\O_X$ where $L=\cup c_1(\L)$, which becomes an isomorphism in $\Mod(\O_X)_\Q$. Let $\beta$ be such that for each $i$, the kernel and cokernel of $L^i:R^{d-i}f_*\O_X\to R^{d+i}f_*\O_X$ are killed by $p^\beta$. We also have a map $L^i:\Fs^{d-i}\to \Fs^{d+i}$ obtained from pulling back the cup-product $\cup c_1(\L)^i:\widetilde{S}(R^{d-i}f_*\O_X)\to\widetilde{S}(R^{d+i}f_*\O_X)$. Then the following diagram commutes:
  \begin{center}
    \begin{tikzcd}
      \Fs^{d-i}\arrow[d]\arrow[r,"L^i"]&\Fs^{d+i}\arrow[d]\\
      R^{d-i}f_*\O_X\arrow[r,"L^i"]& R^{d+i}f_*\O_X.
    \end{tikzcd}
  \end{center}

  Now consider a map $x:\Spec k[[t]]\to S$. Denote the pullback of $f$, $X_x\to \Spec k[[t]]$ again by $f$. By the explicit description of $\Fs^i$ and the functoriality of cohomology, we have a map $x^*\Fs^i\to R^i f_*\O_{X_x}$. This map, by the basechange theorem for crystalline cohomology (\cite{BO:criscohom}[Theorem 7.8]) and a Tor spectral sequence, has kernel and cokernel killed by a power of $p$ depending only on $d$ and $\alpha$. Now, $\L$ pulls back to a line bundle on $X_x$ and thus gives us a map $L^i:R^{d-i}f_*\O_{X_x}\to R^{d+i}f_*\O_{X_x}$, and we have the following diagram commutes:
  \begin{center}
    \begin{tikzcd}
      x^*\Fs^{d-i}\arrow[d]\arrow[r,"L^i"]&x^*\Fs^{d+i}\arrow[d]\\
      R^{d-i}f_*\O_{X_x}\arrow[r,"L^i"]& R^{d+i}f_*\O_{X_x}.
    \end{tikzcd}
  \end{center}
Therefore, the map $L^i:R^{d-i}f_*\O_{X_x}\to R^{d+i}f_*\O_{X_x}$ has kernel and cokernel killed by a power of $p$ depending only of $d$, $\alpha$, and $\beta$.

For a section $y$ of $R^if_*\O_{X_x}$, we have $\phi Ly=pL\phi y$. Using the notation of Remark \ref{rmk:Fstruct} and Definition \ref{defi:Fstruct}, we may write this as $R^{d-i}f_{F,*}\O_{X_x}(i)\to R^{d+i}f_{F,*}\O_{X_x}$ (recall $R^i f_{F,*}\O_X$ is $\Rb f_*\O_X$ with the canonical $F$-structure). Then by \cite{Deligne:lefschetz}, there is an isomorphism $(\Rb_\crys f_{F,*} \O_X)_\Q\to \bigoplus_{i}(R^i_\crys f_{F,*}\O_X)_\Q$. In fact, examining the proof in \cite{Deligne:lefschetz}  we see we even get a map $\Rb_\crys f_{F,*} \O_X(-m')\to \bigoplus_{i}R^i_\crys f_{F,*}\O_X(-m')$ for some $m'$ depending only on $d$ and whose cone has cohomology groups killed by a power of $p$ depending only on the power of $p$ killing the kernel and cokernel of $\Rb^{d-i}f_{F,*}\O_{X_x}\to \Rb^{d+i}f_{F,*}\O_{X_x}(-i)$, that is depending only on $d$, $\alpha$, and $\beta$. This gives $p^{m'}\phi$ equivariant map $\Rb_\crys f_{F,*} \O_X\to \bigoplus_{i}R^i_\crys f_{F,*}\O_X$ whose cone has cohomology groups killed by a power of $p$ depending only on $d$, $\alpha$, and $\beta$.

Now let $T=\Spf W(k)[[t]]$ by which we mean the formal scheme $\varinjlim_k \Spec W(k)[[x]]/p^k$. Give $W(k)[[t]]$ the PD structure compatible with the canonical one on $W(k)$. Recall that for any $U\subseteq \Spec k[[t]]$ and PD-thickening $U\hookrightarrow V$, there is map $V\to T$ compatible with the maps from $U$ and PD-structures (\cite{Katz:Fcrys}).

For any $x:\Spec k[[t]]\to S$, we may find a lift $g:T\to \widetilde{S}$ compatible with PD-structure. Then $(x^*\Fs^i) (T)=g^*\widetilde{S}(R^i_\crys f_*\O_X)$. We then get a map $(x^*\Fs^i)(T)\to \O_T^{r_i}$ with kernel and cokernel killed by $2\alpha$ given by the pullback of $\widetilde{S}(R^i_\crys f_*\O_X)\to \O_{\widetilde{S}}^{r_i}$. Set $\Gs^i=x^*\Fs^i$. This is a crystal as $\Fs$ was a crystal.

  Let form $\Gs^{i,\vee\vee}$ applying Remark \ref{rmk:dualcrystals} twice, and consider the map $\Gs^i\to \Gs^{i,\vee \vee}$. As $T(\Gs^{i\vee \vee})$ is a reflexive module over a two dimensional, regular, noetherian, local ring it is free. Hence $\Gs^{i,\vee \vee}$ is a flat crystal. As $T(\Gs^i)$ maps to a free module with kernel and cokernel killed by $2\alpha$, we conclude that the kernel and cokernel of $T(\Gs^i)\to T(\Gs^{i\vee\vee})$ are killed by a power of $p$ depending only on $\alpha$. A Tor and Ext calculation then implies that the kernel and cokernel of $\Gs^i\to \Gs^{i\vee\vee}$ are killed by a power of $p$ depending only on $\alpha$. Let $C^i_x=\Gs^{i\vee\vee}$. From the map $\Gs^i=x^*\Fs^i\to R^i_\crys f_*(\O_{X,x})$ whose kernel and cokernel are killed by a power of $p$ dependingly only on $d$ and $\alpha$, we see there is map $R^i_\crys f_*(\O_{X,x})\to C^i_x$ with kernel and cokernel killed by a power of $p$ depending only on $d$ and $\alpha$.

  From the maps $R^i_\crys f_*(\O_{X,x})\to C^i_x$, we may find an integer $m>m'$ dependingly only on $d$ and $\alpha$, and maps $p^{m}\phi:C^i_x\to C^i_x$ making the following commutes:
 \begin{center}
    \begin{tikzcd}
      F^{-1}R_\crys^i f_*\O_X\arrow[r,"p^{m} \phi"]\arrow[d]& R_\crys^i f_*\O_X\arrow[d]\\
      F^{-1}C^i_x\arrow[r,"p^{m} \phi"]& C^i_x.
    \end{tikzcd}
  \end{center}
This gives an $F$-structure to $C^i_x$

  Finally we consider the composite \[\Rb_\crys f_*\O_{X_x}\to \bigoplus_i R^i_\crys f_*(\O_{X_x})[-i]\to \bigoplus_i C^i_x[-i],\]
  which is $p^m\phi$ equivariant and has cone with cohomology sheaves killed by a power of $p$ depending only on $d$, $\alpha$, and $\beta$. As $C^i_x$ are flat crystals, this proves the lemma.

\end{proof}

We record the following interesting corollary
\begin{cor}\label{cor:fs}
  Then there an integer $N$ such that for all integers $a$ and $b$ and for each map $x:\Spec k[[t]]\to S$, if $f$ also denotes the map $f:X_x\to \Spec k[[t]]$, then the map $  H^2_\crys (X_x/W(k))^{\phi^k=p^\ell}\to H_\crys^0(\Spec k[[t]], R^2_\crys f_* \O_{X_x})^{\phi^k=p^\ell}$ has cokernel killed by $p^N$.
\end{cor}

\begin{proof}[Proof of Corollary \ref{cor:fs}]
  Now take $x:\Spec k[[t]]\to S$. Let $C_x^i$ and $m$ be as in Proposition \ref{prop:fs}. Let $D_x$ be the cone of $\Rb_\crys f_* \O_{X,x}\to \bigoplus_i C^i_x[-i]$.

  The canonical map \[H^2_\crys(X_x/W(k))=\Hb^2_\crys(\Spec k[[t]],\Rb_\crys f_* \O_{X,x})\to \Hb^2_\crys(\Spec k[[t]],C^i_x[-i])=\bigoplus_{i=1}^2 H^i_\crys(\Spec k[[t]],C_x^i)\]
  has kernel and cokernel killed by a power of $p$ depending only on $m$ and $d$.

  As the above map is $p^m\phi$ linear, for any $a$ and $b$ the map \[H^2_\crys(X_x/W(k))^{p^m\phi^a=p^{b+m}}\to\bigoplus_{i=1}^2 H^i_\crys(\Spec k[[t]],C_x^i)^{p^m\phi^a=p^{b+m}}\] has kernel and cokernel killed by a power of $p$ dependinig only on $m$ and $d$.

  Now the map $H^2_\crys(X_x/W(k))\to H^0_\crys(\Spec k[[t]],C_x^2)$ factors through $H^0_\crys(\Spec k[[t]],R^2_\crys f_*\O_X)$, and the map $H^0_\crys(\Spec k[[t]],R^2_\crys f_*\O_X)\to H^0_\crys(\Spec k[[t]],C_x^2)$ has kernel and cokernel killed by a power of $p$ depending only on $m$ and $d$.

  We thus conclude that $H^2_\crys(X_x/W(k))^{p^m\phi^a=p^{m+b}}\to H^0_\crys(\Spec k[[t]],R^2_\crys f_*\O_X))^{p^m\phi^a=p^{m+b}} $ has cokernel killed by a power of $p$ depending only on $m$ and $d$.

  We conclude that $H^2_\crys(X_x/W(k))^{\phi^a=p^{b}} \to H^0_\crys(\Spec k[[t]],R^2_\crys f_*\O_X))^{\phi^a=p^{b}}$ has cokernel killed by a power of $p$ depending only on $m$ and $d$. Define $N$ so that $p^N$ is this power of $p$.

\end{proof}

\begin{lem}\label{lem:katzlemma}

  Let $k$ be a algebraically closed field of characteristic $p>0$ and $C$ be a flat, coherent crystal on $\Spec k[[t]]$ with constant Newton polygon. Let $C_0$ denote be the $W(k)$ module corresponding to the fiber at the closed point. There is an $n_C$ depending only on the Newton polygon on $C$ such that torsion in the cokernel of the map $H^0_{cris}(\Spec k[[t]],C)\to C_0$ is killed by $p^{n_C}$.
\end{lem}
\begin{proof}
  By \cite{Katz:Fcrys} Theorem 2.7.4 and $C$ is isogenous to a flat, coherent crystal $C'$ which becomes constant over the perfection of $k[[t]]$, and an inspection of the  proof of that theorem shows that the degree of this isogeny is bounded by a power of $p$ which depends only of the Newton polygon of $C$.

  Now let $C'$ be a crystal that becomes isogenous a constant one over the perfection of $k[[t]]$. By the description of crystals in \cite{Katz:Fcrys} Section 2.4, $C'$ corresponds to a flat coherent module $E$ on $W(k)[[t]]$ (hence free) with a nilpotent connection, and $H^0(\Spec k[[t]],C')$ corresponds to flat sections of $C'$ with respect to this connection. Let $C_0'$ be the fiber of $C'$ at $t=0$. Choose a lift of the absolute Frobenius on $k[[t]]$ to $W(k)[[t]]$ which is compatible with the Frobenius on $W(k)$ (for instance that which takes $t$ to $t^p$). Let this lift be denoted $F$. Since $E$ becomes constant on $k[[t]]^{perf}$, by \cite{Katz:Fcrys} Section 2.4, it follows that the extension of scalars of $E$ to $(\varinjlim_F W(k)[[t]])^\wedge=W(k[[t]]^{perf})$ is constant (where we take the $p$-adic completion of the direct limit). Now $F:W(k)[[t]]\to W(k)[[t]]$ has the property that if $x\in W(k)[[t]]$ is not divisible by $p$, then $F(x)$ is not divisible by $p$. Thus a flat section of $E$ cannot become divisible when base changed to $W(k[[t]]^{perf})$. We have then that the cokernel of $H^0(C')\to C'_0$ is $p$ torsion free. Thus the $p$-torsion in the cokernel of $H^0(C)\to C$ is killed by the degree of the isogeny $C\to C'$, which is a power of $p$ bounded by a constant depending only on the Newton polygon of $C$.
\end{proof}

\begin{proof}[Proof of Theorem \ref{thm:2} in the projective case]

  Let $d$ be the relative dimension of $X\to S$. Let $U$ be a dense open of $S$ be such that $R^2f_*\O_X$ is isogenous to a flat crystal, $\C$ (see Lemma \ref{lem:flatonopen}) and furthermore such that $\C$ has constant Newton polygon (see \cite{Katz:Fcrys}).

Choose a point $x:\Spec k[[t]]\to U$. Let $Y=\times_{U,x} \Spec k [[t]]$. Let $Y_0=Y\otimes_{k[[t]]}k$. Take $s$, $m$, $C^i_x$, and $p^m\phi$ as in Proposition \ref{prop:fs}. To ease notation, we will write $C^i$ for $C^i_x$, and $C^i_0$ for  the fiber of $C^i$ at the chosed point of $k[[t]]$. Note as $C^2$ and $x^*\C$ are isogenous, the Newton polygon of $C^2$ is also constant.

We begin with the $p^{m}\phi$ equivariant map $\Rb f_*\O_Y\to \bigoplus_i C^i[-i]$ whose cone has cohomology sheaves killed by $p^s$. Pulling this back to $\Spec k$, and using the basechange theorem for crystalline cohomology (\cite{BO:criscohom}[Theorem 7.8]), we get a $p^m\phi$ equivariant map $\Rb f_*\O_{Y_0}\to \bigoplus_i C^i_0[-i]$ whose cone has cohomology sheaves killed by a power of $p$ depending only on $s$ and $d$. Thus the maps $H^i(Y_0/W(k))\to C^i_{0}$ have kernel and cokernel killed by a power of $p$ depending only on $s$ and $d$.

We then consider the commutative diagram:
\begin{center}
  \begin{tikzcd}
    H^1_\cont(Y,W\Omega^1_{Y,\log})\arrow[r]\arrow[d]& H^2_\crys(Y/W(k))^{p^{m}\phi=p^{m+1}}\rar \dar & H^0(\Spec k[[t]],C^2)^{p^{m}\phi=p^{m+1}}\dar\\
    H^1_\cont(Y_0,W\Omega^1_{Y_0,\log})\rar&H^2_\crys(Y_0/W(k))^{p^{m}\phi=p^{m+1}}\rar&(C^2_{0})^{p^{m}\phi=p^{m+1}}.
  \end{tikzcd}
\end{center}

The top composite has cokernel killed by a power of $p$ depending only on $s$, $d$, and $m$. The bottom composite has kernel and cokernel killed by a power of $p$ depending only on $s$, $d$, and $m$. Both of these follow from Theorem \ref{thm:logcris} and Proposition \ref{prop:fs} (made more explicit in the proof of Proposition \ref{prop:fs}). Lastly, Lemma \ref{lem:katzlemma} implies $p$ torsion in the cokernel of the right vertical arrow is killed by a power of $p$ depending on the Newton polygon of $C^2$ which is the same as that of $\C$. Therefore, there is an $N$ such that \[H^1_\cont(Y,W\Omega^1_{Y,\log})\to H^1_\cont(Y_0,W\Omega^1_{Y_0,\log})\] has cokernel whose $p$-power torsion is killed by $p^N$ and $N$ depends only only on $s$, $d$, $m$, and the Newton polygon on $\C$.

Now let $\L$ be a line bundle on $Y_0$, and let $n\geq N$ be an integer such that $p^{n}\L$ lifts to $Y$. Then $p^{n}c_1(\L)\in H^1_\cont(Y_0,W\Omega^1_{Y_0,\log})$ is in in the image of $H^1_\cont(Y,W\Omega^1_{Y,\log})$. As the $p$-power torsion in the cokernel of $H^1_\cont(Y,W\Omega^1_{Y,\log})\to H^1_\cont(Y_0,W\Omega^1_{Y_0,\log})$ is killed by $p^N$, this implies that $p^Nc_1(\L)$ is in the image of $H^1_\cont(Y,W\Omega^1_{Y,\log})$. By Theorem \ref{thm:vartate}, this means $p^N\L$ lifts to $Y$. Thus, $N$ has the desired property and the theorem is proved.

\end{proof}

\begin{proof}[Proof of Theorem \ref{thm:2prime}]
  We keep notation as in the proof of Theorem \ref{thm:2}, where $U$ and $\C$ have the same hypotheses. Additionally, shrink $U$ so that $X\times_SU\to U$ satisfies Assumption \ref{ass:1} with respect to an integer $N'$, which we can do by Proposition \ref{prop:ass1holds}. Take $x:\Spec k[[t]]\to U$, and $Y=\Spec k[[t]]\times_{x,U}X$, and $Y_m=Y\otimes_{k[[t]]}k[t]/t^{m+1}$. Thus there is an $N$ such that the cokernel of $H^1_\cont(Y,W\Omega^1_{Y,\log})\to H^1_\cont(Y_0,W\Omega^1_{Y_0,\log})$ has $p$-power torsion killed by $p^N$. By Proposition \ref{prop:surj}, $H^1_\cont(Y,W\Omega^1_{Y,\log})\to \varprojlim_m H^1_\cont(Y_0,W\Omega^1_{Y_m,\log})$ is surjective, so the map $\varprojlim_m H^1_\cont(Y_m,W\Omega^1_{Y_m,\log})\to  H^1_\cont(Y_0,W\Omega^1_{Y_0,\log})$ also has $p$-power torsion killed by $p^N$.

  Assumption \ref{ass:1} automatically holds $Y\to \Spec k[[t]]$, as Assumption \ref{ass:1} is stable under pullbacks. Now assume $n\geq N$ is an integer such that $p^{n}y$ lifts adically to $Y$. By Proposition \ref{prop:compatible-padic-bundles}, there exist $z_m\in \Pic(Y_m)\otimes \Z_p$ forming a compatible system, such that the image of $z_0$ in $\NS(Y_0)\otimes \Z_p$ is $p^{n+N'}y$.

 Taking Chern classes, we get an element $c'=\{c_1(y_m)\}_m\in \varprojlim_m H^1_\cont(X_m,W\Omega^1_{X_m})$. The image of $c'$ in $H^1_\cont(Y_0,W\Omega^1_{Y_0,\log})$ is $p^{n+N'}c_1(\L)$. As the $p$-power torsion in the cokernel of $\varprojlim_m H^1_\cont(Y_m,W\Omega^1_{Y_m,\log})\to  H^1_\cont(Y_0,W\Omega^1_{Y_0,\log})$ is killed by $p^N$, there is an element $c$ in the group $\varprojlim_m H^1_\cont(Y_m,W\Omega^1_{Y_m,\log})$ whose image in $H^1_\cont(Y_0,W\Omega^1_{Y_0,\log})$ is $p^{N}c_1(\L)$.

 Now we apply Theorem 3.4(i) of \cite{Morrow:KlogHW} (a variant of Theorem \ref{thm:vartate}), to conclude that $p^{N}\ell$ lifts adically. Thus, $N$ has the desired property and the theorem is proved.

\end{proof}

\section{Proof of Theorem \ref{thm:3a} in the projective case}\label{sec:proofproj}

\begin{setup}\label{set:3b}
  Let $K=k((t))$ and $\O_K=k[[t]]$. Let $\ov{K}$ be an algebraic closure of $K$ and $\O$ be the integral closure of $k[[t]]$ inside $\ov{K}$.

  If $\X\to \B$ is a smooth, proper map of $\O$-schemes with $\B$ irreducible and $\eta$ a generic point of $\B$, let $\B(K)_{\textup{jumping}}=\{b\in \B(K):\rho(\X_b)> \rho(\X_{\ov{\eta}})\}$, and $\B(\O)_{\textup{jumping}}=\B(\ov{K})_{\textup{jumping}}\cap \B(\O)$.
\end{setup}
\begin{thm}\label{thm:3b} 
  Let $\X\to \B$ be a smooth map of smooth, finite-type $k[[t]]$-schemes satisfying Assumption \ref{ass:1} and Assumption \ref{ass:2} and such that $\B$ is irreducible. Then $\B(\O)_{\textup{jumping}}\subseteq \B(\O )$ is nowhere dense, with notation as in Setup \ref{set:3b}
\end{thm}
\begin{proof}[Proof of Theorem \ref{thm:3a} in the projective case]
  Let us take notation of Theorem \ref{thm:3a}. Then by Proposition \ref{prop:ass1holds} and Theorem \ref{thm:2prime}, there exists a dense open $U\subseteq \B$ on which Assumption \ref{ass:1} and Assumption \ref{ass:2} hold. For all $c$, $U(\O_c)_{{\text{jumping}}}$ is nowhere dense in $U(\O_c)$ by Theorem \ref{thm:3b} as $\X_U\otimes \O_c\to U\otimes \O_c$ satisfies Assumptions \ref{ass:1} and \ref{ass:2}. This proves the theorem. 
\end{proof}

\begin{proof}[Proof of Theorem \ref{thm:3b}]
    Let $N$ be such that $X\to \B$ satisfies Assumption \ref{ass:1} and Assumption \ref{ass:2} with respect to $N$.

  If $\B'\to \B$ is surjective and \'{e}tale and the theorem for $\X\times_\B \B'\to \B'$, then the theorem holds for $\X\to \B$, because \'{e}tale maps induce a surjective local homeomorphism $\B'(\O)\to \B(\O)$. As $\X\to \B$ is smooth, it has a section \'{e}tale locally. Thus we may and do assume that $\X\to \B$ has a section.

We must prove that the complement of the locus of $x\in B(\O)$ such that $\NS(\X_{\ov{\eta}_x})\otimes \Q\to \NS(\X_{s_x})\otimes \Q$ is an isomorphism is nowhere dense.

Pick a point $b\in \B(k)$. Let $V$ be a nonempty open subset of the $\B(\O)$ contained in the residue disk of points reducing to $b$. As $\B$ is smooth, we may find a finite extension $K'$ of $K$, and power series ring $R=\O_{K'}[[x_1,\ldots,x_n]]$ with $\O_K$-morphism $\Spec R\to \B$ such that the $\O$-points of $\Spec R$ bijectively to the $\O$-points of an open subset of $V$ (see Section \ref{sec:piadic}). Give the $\O$-points of $\Spec(R)$ the subspace topology.

Now replace $K$ by $K'$. Let $B=\Spec R$ and $X=\X\times_\B B$. It suffices to show that there is an open subset $U$ of $B(\O)$ such that for all $x\in U$, $\NS(X_{\ov{\eta}_x})\otimes \Q\to \NS(X_{s_x})\otimes \Q$ is an isomorphism. Denote the fiber of $X$ over the closed point of $B$ by $X_0$. Let $\m_R$ be the maximal ideal of $R$.

Let $\ov{\eta}$ be a geometric generic point of $B$. Let $G$ be the cokernel of the map $\NS(\X_{\ov{\eta}})\to \NS(X_0)$. This is a finitely generated abelian group. Let $M$ be the size of the torsion subgroup of $G$, and recall that $X\to \B$ satisfies Assumption \ref{ass:2} with respect to $N$. This makes $M p^N G$ a finitely generated, torsion free abelian group. Let $\ov{\ell_1},\ldots,\ov{\ell_n}$ be a basis for this group, and lift each $\ov{\ell_i}$ to an element of $\ell_i$ of $M p^N \NS(X)$. 

\begin{claim}
If $U\subseteq B(\O)$ is such that for all $x\in U$, no $\Z$-linear combination of the $\ell_i$ is in the image of $\NS(X_{\ov\eta_x})\to \NS(X_0)$, then for all $x\in U$, $\NS(X_{\ov{\eta}_x})\otimes \Q\to \NS(X_{s_x})\otimes \Q$ is an isomorphism.
\end{claim}

Assume then that there is a nonempty open set $U\subseteq B(\O)$ such that for each $x\in U$, no $\Z$-linear combination of $\ell_i$ lifts to $\O$. For $x\in U$, we have a sequence of injections \[\NS(X_{\ov\eta})\otimes \Q\hookrightarrow \NS(X_{\ov{\eta_x}})\otimes \Q\hookrightarrow \NS(X_0).\] Now, by construction the cokernel of $\NS(X_{\ov\eta})\otimes \Q \to \NS(X_0)\otimes \Q$ is freely generated as a $\Q$ vector space by the images of the $\ell_i$. If no $\Z$-linear combination of the $\ell_i$ is in the image of $ \NS(X_{\ov{\eta_x}})\hookrightarrow \NS(X_0)$, no $\Q$-linear combination of the $\ell_i$ will be in the image of $ \NS(X_{\ov{\eta_x}})\otimes \Q \hookrightarrow \NS(X_0)\otimes \Q$. It follows, in this case that $\NS(\X_{\ov\eta})\otimes \Q\hookrightarrow \NS(\X_{\ov{\eta_x}})\otimes \Q$ surjective and hence an isomorphism, thereby proving the  claim.

Thus, our goal is now to show that there exists an open $U\subseteq B(\O)$ such that for any $x\in U$ no $\Z$-linear combination of the $\ell_i$ is in the image of $\NS(X_{\ov\eta_x})\to \NS(X_0)$.

As Assumption \ref{ass:1} is true in our setup and we reduced to the case when $\X\to \B$ had a section, Proposition \ref{prop:lifting-bundles} implies that it suffices to show that there exists an open $U\subseteq B(\O)$ such that for any $x\in  U$, no $\Z$-linear combination of the $\ell_i$ lifts adically to $X_x$.

Assume for sake of contradiction, that for every nonempty open $U\subseteq B(\O)$ that there exists a $x\in U$ and nonzero $\ell=\sum_i a_i\ell_i\in \NS(X_0)$ such that $\ell$ lifts adically to $X_x$. 

We will now argue by induction by induction on $m$, that there are elements $\{\ell_{m,i}\}_{i=1}^n\subseteq \sum_i \Z\ell_i\subseteq \NS(X_0)$ such that $\sum_{i}\Q\ell_{m,i}=\sum_i\Q\ell_i\subseteq \NS(X_0)\otimes \Q$ and a nonempty open $U_m\subseteq B(\O)$ such that for every $x\in U_m$ and every nonzero $\ell \in \sum_{i=1}^m\Z_p \ell_i\subseteq \NS(X_0)\otimes \Z_p$, $\ell$ does not lift adically to $X_x$.

The base case of $m=0$ is trivial taking  $U_0=B(\O)$ and we clearly have a contradiction if we manage to prove the induction hypothesis for $m=n$. Thus we assume the induction hypothesis is true for $0\leq m<n$ and we will prove it is true for $m+1$. Thus let $U_m$ and $\{\ell_{m,i}\}_{i=1}^n$ be as in the induction hypothesis for $m$.

Again we may find a finite extension $K'$ of $K$, and power series ring $R'=\O_{K'}[[x_1,\ldots,x_n]]$ with $\O_K$-morphism $\Spec R'\to \B$ such that the $\O$-points of $\Spec R$ bijectively to the $\O$-points of an open subset $B'$ of $U_m$. Let $\mathfrak{m}_{R'}$ be the maximal ideal of $R'$. We replace $U_m$ to $B'$. Let $X'=X\times_B B'$ and  $X'_n=X'\otimes R'/\mathfrak{m}_{R'}^{n+1}$.

By the hypothesis to be contradicted, we may find $y\in B'$ and $\ell=\sum_ia_i\ell_i\neq 0 $ such that $\ell$ lifts adically to $X_y$. Replacing $\ell$ with a some multiple, we may assume that $\ell=\sum_i b_i \ell_{m,i}$ for some $b_i\in \Z$ and that $\ell$ lifts to $X_x$. By the induction hypothesis, it must be true that $b_i\neq 0$ for some $i>m$. Now, we find $\{\ell_{m+1,i}\}_{i=1}^n\subseteq \sum_i \Z\ell_{m,i}$ such that $\ell_{m+1,i}=\ell_{m,i}$ for $i\leq m$, $\ell_{m+1,i}=\ell$, and such that the $\{\ell_{m+1,i}\}_{i=1}^n$ span $\sum_i \Q\ell_{m,i}$ rationally.

Let $W=\sum_{i=1}^{m+1}\Z\ell_{m+1,i}\subseteq \NS(X_0)$. For $n\geq 0$, consider $W/p^n W$. Let $V_n$ be the subset of $W/p^n W$ consisting of elements whose image in $W/pW$ is not zero. For each $u\in V_n$, consider $X_n'/B_n'$ and $\Pic^u_{X'_n/B'_n}$ the component of the Picard scheme corresponding to a lift of $u$ to $U_{m+1}$. This is independent of lift of $u$ up to isomorphism as the $p^n$ power of any line bundle on $X_0'$ lifts to $X_n'$ by Remark \ref{rmk:lifting-p-power}.

\begin{claim}
  There exists an $n$ such that for each $u\in V_n$, the scheme theoretic image of $\Pic^u_{X_n'/B_n'}\to B'_n$ is not $B'_n$.  
\end{claim}

Assume for sake of contradiction that for each $n$ that there exists a $u\in V_n$, such that the scheme theoretic image of $\Pic^u_{X_n'/B_n'}\to B_n'$ is $B_n'$. By compactness of $\Z_p^n-p\Z_p^n$, we may find a system $u_n \in V_n$ such that the scheme theoretic image of $\Pic^{u_n}_{X_n'/B_n'}\to B_n'$ is $B_n'$, and the image of $u_{n+1}$ in $V_n$ is $u_n$. The system $\{u_n\}$ corresponds to a nonzero object in $W\otimes \Z_p$. Let $\rho$ denote the image of this object in $\NS(X_0)\otimes \Z_p$.  By Remark \ref{rmk:p-adic-Picard}, we can form the formal scheme $\Pic^\rho_{X/B}$. Call this formal scheme $\Ps$. We then have that $\Ps\to \B$ is scheme theoretically surjective.

By Proposition \ref{prop:liftingformal}, this implies that for every $x\in V(\O)$, the $p$-adic line bundle $\rho$ to lifts to $X_x$. If $\rho\in \sum_{i=1}^m\Z_p\ell_{m,i}=\sum_{i=1}^m\Z_p\ell_{m+1,i}\subseteq W\otimes \Z_p$, this gives a contradiction of the induction hypothesis. Assume next that $\rho\in \Z_p \ell_{m+1,m+1}$. By Proposition \ref{prop:lifting-bundles} as $X/B$ satisfies Assumption \ref{ass:1} with respect to $N$, for any $x\in V(\O)$ we have that $p^N\ell_{m+1,m+1}$ lifts along $x$. This would imply that any component of $\Pic_{X/B}$ corresponding to $p^N\ell_{m+1,m+1}$ surjects onto $B$. This is a contradiction as it would imply that that $p^N\ell_{m+1,m+1}$ lifts along the geometric generic fiber of $B$.

Lastly, if $\rho\in W\otimes \Z_p-(\sum_{i=1}^m\Z_p\ell_{m+1,i})-\Z_p\ell_{m+1,m+1}$, then $\rho$ and $\ell_{m+1,m+1}$, generate together a nonzero element of $W\otimes \Z$ and thus some nonzero element of $W\otimes \Z$ lifts adically to $X_y$ , which is a contradiction (recall $y\in U_m$ was such that $\ell_{m+1}$ lifted adically to $X_y$). Thus we have proved the claim.

Now choose $n$ such for each $u\in V_n$, the scheme theoretic image of $\Pic^u_{X_n'/B_n'}\to B_n'$ is not $B_n'$. 

Consider $\Res_{k[t]/t^{n+1}/k}B_n$. There is a closed proper subset of this scheme corresponding to Weil restriction of the closed subschemes given by the scheme theoretic images of the $\Pic^u$ for $u\in V_n$. Choose a point, $p$, that is not in that closed subscheme. There is an open subset $P\subseteq B(\O)$ corresponding to points of $B(\O)$ whose reduction mod $t^{n+1}$ agrees with $p$. 

\begin{claim}
  If $z\in  P$, no nonzero element of $\sum_{i=1}^{m+1}\Z_p \ell_{m+1,i}$ lifts adically to $X_z$
\end{claim}

Assume for some $z\in P$, $\sum_{i=1}^{m+1} a_i\ell_{m+1,i}\neq 0$ lifts adically to $X_z$ and Assumption \ref{ass:1} is true with respect to $N$, Proposition \ref{prop:compatible-padic-bundles} implies $\sum_{i=1}^{m+1} p^Na_i\ell_{m+1,i}$ lifts to $X_z$.

Now as every $\ell_{m+1,i} \in Mp^{N} \NS(X_0)$ and Assumption $2$ with respect to $N$ holds, we see that some other $\sum_{i=1}^m b_i\ell_{m+1,i}$ lifts to $X_z$ where at least one $b_i$ is not divisible by $p$. This is impossible by construction of $P$. This proves the claim.

But we had a assumed by sake of contradiction that no such a $P$ could exist. Thus, there must exist a $U_{m+1}$ satisfying the induction hypothesis. Thus the induction is complete, and we have finished the proof of the the theorem.

\end{proof}

\section{Reductions}
\subsection{Addendum on $t$-adic topology}\label{sec:appen}
For this section let $K=k((t))$, and let $\ov{K}$ be an algebraic closure with its canonical topology. Let $X$ and $Y$ be smooth finite type $\ov{K}$-schemes, and let $f:X\to Y$ be a finite morphism. We will show that there is a dense Zariski open $U\subseteq X$ such that $f:(f^{-1}U)(\ov{K})\to U(\ov{K})$ is a local homeomorphism (for our purposes a map from an empty scheme is a local homeomorphism).

First, we may choose $U$ so that $f^{-1}U\to U$ is finite and flat. In fact we may even choose $U$ so that $f^{-1}U$ is a finite \'{e}tale morphism followed by a finite purely inseparable morphism. The desired statement is well-known the the \'{e}tale case, so we must only consider the finite inseparable case. Shrinking $U$, we may assume that that $f^{-1}U\to U$ is a composition $S_n=f^{-1}U\to S_{n-1}\to\cdots S_1\to S_0=U$, where $S_{i+1}=S_i[f_i^{1/p}]$ for some function $f_i$ on $S_i$. We may also assume that each $S_i$ is smooth again by potentially shrinking $U$ (note as $\ov{K}$ is algebraically closed, every reduced finite type scheme over it has a dense open subset on which it is smooth). Furthermore, we can clearly assume that the $f_i$ are not $p$ powers. In this way, we reduce to the case when $f^{-1}U=U[f^{1/p}]$ for some function $f$ on $U$ that is not a $p$th power. Since $U$ is smooth and $f$ is not a $p$th power, $df\neq 0$, thus potentially shrinking $U$ again, there is an \'{e}tale map $U\to \A^n$, where if we give $\A^n$ the coordinates $x_1,\ldots,x_n$, the function $x_1$ pulls back to $f$. As \'{e}tale maps induce local homeomorphisms on $\ov{K}$-points, we may reduce to the case when $Y=U\cong \A^n$ and $X=\A^n$ and $f:X\to Y$ is the map $(x_1,\ldots,x_n)\mapsto (x_1^p,\ldots, x_n)$, and for this case the claim is clear.
\subsection{Deducing the proper case from the projective case}\label{sec:projprop}

Now let $k$ be an algebraically closed field. Let $S$ be a smooth, finite type, connected $k$-scheme, and let $\X\to S$ be a proper map. Let $\eta$ be the generic point of $S$ and $k(\eta)$ the function field of $S$. Let $\ov{k(\eta)}$ be an algebraic closure of $k(\eta)$. Denote $X$ the fiber of $\X$ over $\ov{k(\eta)}$

By the alterations theorem, there is an alteration $f:Y\to X$ of $\ov{k(\eta)}$ schemes, from a regular, projective $\ov{k(\eta)}$-scheme $Y$. Let $X'$ be the normalization of $X$ in the function field of $Y$. Thus we have a factorization $Y\to X'\to X$. The map $X'\to X$ is finite. As $X'$ is normal, it is regular away from a closed set of codimension $2$. We have $X'\to X$ is then flat on the locus where $X'$ is regular, as may be seen for instance by using local criterion for flatness and a Koszul sequence argument (see Theorem 18.16 in \cite{Eisenbud:commalg}) . Thus there exists an open $U\subseteq X$ whose complement has codimension at least $2$ and such that over $U$ the map $X'\to X$ is finite flat. Now, $Y\to X'$ is a birational map of normal $\ov{k(\eta)}$-schemes and thus there is an open $V\subseteq X'$ such that $Y\to X'$ is an isomorphism over $V$. Take a maximal such $V$. The image of $X'-V$ in $X$ will also have codimension $2$. Thus we may shrink $U$ so that the map $Y\to X$ is finite and flat over $U$ and the complement of $U$ has codimension at least $2$. Let $d$ be the degree of $f$ over $U$.

Now consider the map $f^*:\Pic(X)\to \Pic(Y)$. We will describe another map $s:\Pic(Y)\to \Pic(X)$ such that $s\circ f^*$ is multiplication by $d$. First we have the restriction map $r_{f^{-1}U}:\Pic Y\to \Pic f^{-1}U$. Then we have the norm map $N:\Pic f^{-1}U\to \Pic U$. Finally, as the complement of $U$ has codimension at least $2$ and as $X$ is regular, the natural restriction map $r_U:\Pic X\to \Pic U$ is an isomorphism. Set $s:=r_U^{-1}\circ N \circ r_{f^{-1}U}$. That $s\circ f^*$ is multiplication by $d$ can be seen by composing both maps with $r_U$, at which point it is obvious. 

These constructions respect algebraic equivalence, thus induce maps $f^*:\NS(X)\to \NS(Y)$ and $s:\NS(Y)\to \NS(X)$.

Now we spread out $Y$ and $U$. Specifically, we find a smooth $k$-scheme $S'$ with a map $S'\to S$ finite over an open subset of $S$, a smooth $S'$ schemes $\Y$, and a map $f:\Y\to \X\times_S S'$ such that when these data is based changed along a chosen $k(S')\to \ov{k(\eta)}$ recover $Y$ and the map $f:Y\to X$. Furthermore, we may choose the data so that we may find an open subset $\mathcal{U}\subseteq \X\times_S S'$ whose complement has codimension $2$ on every fibre, and such that $\Y\to \X\times_S  S'$ is finite flat of degree $d$ over $\mathcal{U}$.

For any algebraically closed field, $L$, and $t\in S'(L)$, let $\Y_t$ and $\X_t$ be the pullbacks of $\Y$ and $\X$ along $t$. Thus as above we get maps $f^*_t:\Pic (\X_t)\to \Pic(\Y_t)$ and $s_t:\Pic(Y_t)\to \Pic(X_t)$ such that $s_t\circ f^*_x=d$, and similar maps on N\'{e}ron-Severi groups. 

Now let $x:\Spec k[[t]]\to S'$. The maps $f^*$ and $s$ are compatible with specialization, thus we get a commutative diagram
\begin{tikzcd}
  \NS(\X_{\ov{\eta}_x})\arrow[r,"f^*"]\arrow[d]& \NS(\Y_{\ov{\eta}_x})\arrow[r,"s"]\arrow[d]&\NS(\X_{\ov{\eta}_x})\arrow[d]\\
  \NS(\X_{s_x})\arrow[r,"f^*"]&\NS(\Y_{s_x})\arrow[r,"s"]&\NS(\X_{s_x}),
\end{tikzcd}
where the vertical maps are specialization.

Thus if the map $\NS(\Y_{\ov{\eta}_x})\otimes \Q\to \NS(\Y_{s_x})\otimes \Q$ is an isomorphism, then so is $\NS(\X_{\ov{\eta}_x})\otimes \Q\to \NS(\X_{s_x})\otimes \Q$. Thus, using Section \ref{sec:appen} to relate $S'$ and $S$, it follows thatTheorem \ref{thm:3a} in the proper case follows from Theorem \ref{thm:3a} in the projective case. Specifically, if $\B$ is an in the statement of Theorem \ref{thm:3a}, we take $S=\B$ in the above discussion and choose $S'$ to also be smooth over $C$. Then Theorem \ref{thm:3a} for $\mathscr{Y}$ implies Theorem \ref{thm:3a} for $\X$. 

Theorem \ref{thm:2} follows from the projective case using an analogous discussion, but replacing $\NS$ with $\Pic$ in the above commutative diagram. 
\subsection{Proof of Theorem \ref{thm:1}}\label{sec:spread}
\begin{prop}\label{prop:allfields}
  If Theorem \ref{thm:1} is true when $K=\ov{k(t)}$ where $k$ is an algebraically closed field of characteristic $p$, then Theorem \ref{thm:1} is true for all algebraically closed $K$ of characteristic $p$ not isomorphic to $\ov{\F}_p$.
\end{prop}
\begin{proof}
  Let $L=\ov{\F_p(t)}$.
  
  Let $K$ be an algebraically closed field of characteristic $p$ not isomorphic to $\ov{\F}_p$. As in the statement of Theorem \ref{thm:1}, let $X$ and $S$ be $K$-varieties, and let $X\to S$ be a smooth, proper $K$-morphism. As $K$ is perfect, $S$ is smooth on a dense open subset. We may and do replace $S$ with this subset.

  Next, as $K$ is not isomorphic to $\ov{\F}_p$, there is a map of fields $L\to K$ (choose an element of $K$ transcendental over $\F_p$, to get a map $\F_p(t)$, then extend it to a map from the algebraic closure).

  Spreading out, we get a finite type integral $L$-scheme $B$ with an injection of the function field $L(B)\to K$ extending the chosen injection $L\to K$, finite type smooth $B$-schemes $\X$ and $\mathscr{S}$ with $B$-morphism $\X\to \mathscr{S}$, such that base changing $\X\to \mathscr{S}$ to $K$ along $L(B)\to K$ recovers $X\to S$.

  Choose a geometric generic point $\ov{\eta}$ of $S$, which in turn also gives a geometric generic point of $\mathscr{S}$ from the morphism $S\to \mathscr{S}$.

  Now applying Theorem \ref{thm:1} for $L$ to $\X\to \mathscr{S}$ we find a point $x\in \mathscr{S}(L)$ such that $\rho(X_{\ov\eta})=\rho(\X_{\ov{\eta}})=\rho(\X_x)$. Next, as $\mathscr{S}\to B$ is smooth, we may find an integral $L$ scheme $B'$ with \'{e}tale map $B'\to B$ and $B$-morphism $B'\to \mathscr{S}$ such that $x$ is in the image of $B'(L)$. To see this is possible, note if $\mathscr{S}\cong \A^n_B$ then we may even get a section from $B$, and in general there is a Zariski open $U\subseteq \mathscr{S}$ containing the image of $x$ and an \'{e}tale map $\mathscr{S}\to \A^n_B$, and we pull back a section of $B$ containing the image of $x$ in $\A^n_B$.

  Consider the image of $B'$ in $\mathscr{S}$. It is a locally closed subscheme that maps generically finitely to $B$. We get a map $\Spec k(B')\to \mathscr{S}$ from considering the generic point of $B'$. As $k(B')$ is algebraic over $k(B)$, we may choose a map $k(B')\to K$ extending the chosen map $k(B)\to K$. This gives us a map $y:\Spec K\to \mathscr{S}$, which we view as a geometric generic point of the image of $B'$.

  Thus as $y$ is a specialization of $\ov{\eta}$, and $x$ a specialization of $y$, we get the inequalities $\rho(X_{\ov\eta})=\rho(\X_{\ov\eta})\leq \rho(\X_y)\leq \rho(\X_x)$. But $\rho(\X_{\ov\eta})=\rho(\X_x)$, so $\rho(X_{\ov\eta})=\rho(\X_y)$. Now $y:\Spec K\to \mathscr{S}$, gives a map $y':\Spec K\to S$, and $\X_y$ corresponds to $X_{y'}$ under the identification of the pullback of $\X\to \mathscr{S}$ along $\Spec K\to B$ with $X\to S$. Thus we have $\rho(X_{\ov\eta})=\rho(X_{y'})$ completing the proof.
\end{proof}
\begin{proof}[Proof of Theorem \ref{thm:1}]
  By Proposition \ref{prop:allfields}, we may assume that $K=\ov{k(t)}$ where $k$ is an algebraically closed field of characteristic $p$. Let $X\to B$ be a smooth map of finite type $K$-schemes with $B$ integral. We may spread out to get a curve $C$ over $k$, a smooth map of finite type $C$-schemes $\X\to \B$, and if $k(C)$ denotes the function field of $C$ an inclusion $k(C)\hookrightarrow K$ such that base-changing $\X\to \B$ along the associated map $\Spec K\to C$ we recover $X\to B$.

  By Theorem \ref{thm:3a} there is a dense Zariski open $U\subseteq \B$ and a $c\in C(k)$, such that if $\O_c$ is as in Setup \ref{set:1}, then $U(\O_c)_{\text{jumping}}\subseteq U(\O_c)$ is nowhere dense and nonempty. As $U(\O_c)\cap U(\ov{k(C)})$ is dense in $U(\O_c)$, this implies there is a $\B(\ov{k(C)})$-point outside the jumping locus. This thus gives a $K$-point outside the jumping locus proving the theorem.
\end{proof}

\section*{Acknowledgments}
We thank Bjorn Poonen for suggesting this problem and Matthew Morrow for conversations regarding parts of this paper.
\bibliography{ns}
\bibliographystyle{alpha}
\end{document}